\documentclass{article}
\usepackage{graphicx}
\usepackage{amssymb,amsmath}

\usepackage{comment}
\usepackage{enumitem}
\usepackage[margin=1in]{geometry}
\usepackage{color}
\usepackage{amsthm}
\newtheorem{theorem}{Theorem}
\newtheorem{proposition}{Proposition}

\newtheorem{lemma}{Lemma}
\DeclareMathOperator{\E}{\mathbb{E}}
\DeclareMathOperator{\F}{\mathcal{F}}
\DeclareMathOperator{\R}{\mathbb{R}}

\newcommand{\pow}{\mathcal{P}}
\newcommand{\1}{\mathbf{1}}
\renewcommand{\P}{\mathbb{P}}

\def\A{\mathcal{A}}
\def\bf{\mathbf}
\newcommand{\C}{\mathcal{C}}
\newcommand{\Cu}{\mathbf{C}}

\newcommand{\Ru}{\mathbf{R}}
\newcommand{\sgn}{\textrm{sgn}}

\DeclareMathOperator{\var}{Var}
\newcommand{\poi}{\textrm{Poisson}}

\newcommand{\ep}{\epsilon}

\def\beq{ \begin{equation} }
\def\eeq{ \end{equation} }
\def\bay{ \begin{array} }
\def\eay{ \end{array} }
\def\lng{ \langle }
\def\rng{ \rangle }
\def\eps{ \epsilon }
\def\dlt{ \delta }

\title{The Naming Game on the complete graph}
\date{}
\author{Eric Foxall}
\begin{document}
\maketitle 

\begin{abstract}
 We consider a model of language development, known as the naming game, in which agents invent, share and then select descriptive words for a single object, in such a way as to promote local consensus. When formulated on a finite and connected graph, a global consensus eventually emerges in which all agents use a common unique word. Previous numerical studies of the model on the complete graph with $n$ agents suggest that when no words initially exist, the time to consensus is of order $n^{1/2}$, assuming each agent speaks at a constant rate. We show rigorously that the time to consensus is at least $n^{1/2-o(1)}$, and that it is at most constant times $\log n$ when only two words remain. In order to do so we develop sample path estimates for quasi-left continuous semimartingales with bounded jumps.
\end{abstract}



\section{Introduction}
The study of social dynamics from the standpoint of statistical physics is an area which has seen increased attention in recent years \cite{statsoc}. Historically, interacting particle system models of opinion dynamics, such as the voter model, have been of interest to mathematicians and studied in detail. However, new models emerging in the physics literature have yet to be given a fully rigorous mathematical treatment. One of these is a model of language development known as the naming game. This is a simple model of invention, sharing, and selection of words that displays eventual consensus towards a common vocabulary. It has been studied, using numerical simulations and heuristic computations, on lattices \cite{nglatt}, the complete graph \cite{ngcomp} and some random graphs \cite{ngrand}. As a first effort from the standpoint of probability theory, we study the naming game on the complete graph and give rigorous proof of some scaling relations that have been observed numerically.\\

We first recall the definition of the naming game on a general locally finite undirected graph~$G=(V,E)$.
 Individuals correspond to vertices of the graph, and each individual speaks to its neighbours at a certain rate. The idea is that individuals are attempting to agree on a word to describe a certain object, for which initially, no descriptive words exist. The interaction rules are as follows.
\begin{itemize}
\item Speaker:
\begin{itemize}
\item If the speaker does not know a word to describe the object then she invents a word and speaks it to the listener. \vspace*{-4pt}
\item On the other hand, if the speaker does know at least one word to describe the object then she selects a word uniformly at random from her vocabulary and speaks it to the listener. \vspace*{-4pt}
\end{itemize}
\item Listener:
\begin{itemize}
\item If the listener already knows the chosen word, then both speaker and listener delete the remainder of their vocabulary and remember only that word.
\item Otherwise, the listener adds the chosen word to their vocabulary.
\end{itemize}
\end{itemize}
 Thus there is a mechanism both for the creation of new words, and for deletion and eventual agreement upon a single word. We now make this description rigorous.
 The process is denoted $(W_t)_{t \ge 0}$ with $W_t:V \to \pow_o(V)$ for each $t\ge 0$, where $\pow_o(V)$ is the collection of finite subsets of $V$. Thus, for each vertex~$v \in V$, we have a process~$W_t(v)$ whose state space consists of all finite subsets of the vertex set~$V$ and which is defined as
 $$ W_t(v) = \{w \in V : v \ \hbox{knows the word invented by} \ w \}. $$
 The process evolves as follows:
 For each $v \in V$, at the times of an independent Poisson process with rate one, $v$ chooses a listener~$w$ uniformly at random from the set $\{u : uv \in E \}$;
 say this occurs at time $t$.
\begin{itemize}
\item If $W_{t^-}(v)$ is empty then $v$ speaks word $v$ to $w$, so that $W_t(v) = \{v \}$ and $W_t(w) = W_{t^-}(w) \cup \{v \}$. \vspace*{-4pt}
\item If $W_{t^-}(v)$ is non-empty then $v$ chooses a uniform random word $u$ from $W_{t^-}(v)$ and speaks it to $w$.
\begin{itemize}
\item If $u \in W_{t^-}(w)$ then $W_t(v) = W_t(w) = \{u\}$.
\item If $u \notin W_{t^-}(w)$ then $W_t(v)$ is unchanged and $W_t(w) = W_{t^-}(w) \cup \{u\}$.
\end{itemize}
\end{itemize}
 If~$G$ is connected and finite, then with probability one, the system eventually settles into one of the set of absorbing states
 $$ \{W_t(v) = \{w \} \ \hbox{for all} \ v \in V : w \in V \} $$
 and we would like to know what happens on the way to this consensus. Let
 $$V_t = \bigcup_v W_t(v)$$
 denote the set of words in existence at time $t$. If $G$ is the complete graph on $n$ vertices, i.e.,
 $$V=\{1,\dots,n\}\quad \hbox{and} \quad E=\{\{v,w\}:v,w \in V, \ v\ne w\},$$
 numerical studies and heuristic computations \cite{ngcurve} indicate three distinct phases.
 \begin{enumerate}[noitemsep]
 \item Early phase: $V_t$ rises from $0$ to about $n/2$ in about $\frac{1}{2}\log n$ time.
 \item Middle phase: $V_t$ remains fairly constant up till about $n^{1/2}$ time.
 \item Late phase: $V_t$ falls sharply to 1 within about $n^{1/4}$ time.
 \end{enumerate}
In this article we consider the early and middle phases, and what we call the final phase, where we assume that $V_t$ is initially equal to $2$, and track the dynamics until it goes to $1$. The bulk of the late phase, during which the diversity of language collapses from a large number to a small number of different words, is more difficult to assess, and is not considered here.\\

In the next section we construct the model as a stochastic process, then describe the main results and give the layout for the rest of the article.

\section{Construction and Main Results}
 We first note a useful ``graphical construction'' of the process, on a general locally finite graph~$G$, from arbitrary initial data. We assume the vertices are totally ordered according to some fixed order.
 Given $\mu>0$, let~$\{(s_i, u_i) : i \geq 1 \}$ be an independent and identically distributed sequence, with each~$s_i$ exponentially distributed with mean one and each~$u_i$ independent of $s_i$ and uniform on~$[0, 1]$,
 and for~$i \geq 1$, let~$t_i = \mu^{-1}\sum_{j = 1}^i s_j$.
 Then, the set of points
 $$ U := \{(t_i, u_i) : i \geq 1 \} \subset \R_+ \times [0, 1] $$
 defines what we call an augmented Poisson point process with intensity $\mu$, since $(t_i)$ are the jump times of a Poisson process with intensity $\mu$ and each point~$t_i$ comes equipped with an independent uniform random variable~$u_i$ to help with the decision-making process.\\

\indent Let~$F$ denote the set of directed edges~$\{(v, w): vw \in E \}$, and associate to each directed edge~$(v, w) \in F$ an independent augmented Poisson point process~$U (v, w)$ with
 intensity~$(\deg v)^{-1}$.
 Suppose that $(t,u) \in U(v,w)$ and $|W_{t^-}(v)| = k$, with $W_{t^-}(v) = \{w_1,\dots,w_k\}$ labelled in increasing order.
\begin{itemize}
\item If $k=0$ then $v$ speaks word $v$ to $w$ at time $t$.
\item If $k\geq 1$, then $v$ speaks word $w_i$ to $w$ at time $t$ if and only if
$$(i-1)/k \leq u < i/k.$$
\end{itemize}
We then follow the rules as described above to determine $W_t$.
 If~$G$ is a finite graph, then since the intensity of the union $\bigcup_{(v,w) \in F}U (v,w)$ is finite, its points are well-ordered in time with probability 1, and so $W_t$ can be determined from the initial
 state and the points~$U (v, w)$ by updating sequentially in time. If~$G$ is an infinite graph, one needs to ensure that for each spacetime point~$(v,t)$, a finite number of events suffices to determine~$W_t(v)$. Although this is not hard to do, we will ignore it since from here on we focus on the case where $G$ is the complete graph on $n$ vertices and thus finite for any $n$.\\
 
 Recall that $V_t = \bigcup_v W_t(v)$ denotes the set of words in existence at time $t$. The following result gives estimates of $V_t$ in the middle phase of the process.
 
 \begin{theorem}
 \label{thm:midphase}
 For any $\ep>0$, let $a = (\frac{1}{2}+\ep)\log n$ and $b = n^{1/2-\ep}$. Then as $n\to\infty$
 $$\P(\sup_{a \le t\le b}|V_t - \frac{n}{2}|  = o(n)) = 1-o(1).$$
 \end{theorem}
 The result is proved in two main steps.
 \begin{enumerate}[noitemsep]
 \item First, we show that $n/2 + n^{1/2+o(1)}$ are ever created, and within $(\frac{1}{2}+o(1))\log n$ time.
 \item Then, we show that $o(n)$ words are deleted in $n^{1/2-o(1)}$ time.
 \end{enumerate}
 The proof relies on approximating the size of the \emph{cluster} $\C_t(w)$ corresponding to a given word $w$ by a sort of branching process evolving in a non-stationary random environment. The cluster is defined by
 $$\C_t(w) = \{v:w \in W_t(v)\}$$
and is the set of individuals that know word $w$ at time $t$. We also need to control the correlation between distinct clusters $\C_t(w_1,),\C_t(w_2)$. To achieve both tasks we will use a slightly modified graphical construction which is better tailored to tracking the evolution of one or more distinguished clusters.\\
 
 For the next result we introduce some notation. Let~$\Theta_t(W)$ denote the configuration at time~$t$ when the initial configuration is $W$, and let~$V' \subset V$. If~$W:V \to \pow_o(V') \setminus \{\emptyset\}$ then clearly
 $$\Theta_t(W): V \to \pow_o(V') \setminus \{\varnothing\} \quad \hbox{for all} \quad t > 0,$$
 that is, if each vertex has initially a non-empty vocabulary consisting of words in $V'$, the same is true at later times. In particular, if $V'= \{A,B\}$ for a pair of words $A,B$, then each vertex has one of the three types~$A, B$ and~$AB$.
 We note that, starting from~$W(v) = \varnothing$ for all $v$, before the process achieves consensus there is a good chance that at some point only two words remain, so we can think of it
 as the final phase of the process. For the complete graph on $n$ vertices, the rate of change of the number of individuals of each type does not depend on the particular location of the individuals.
 Therefore, letting~$X_t, Y_t, Z_t$ denote the number of sites at time~$t$ with respective types~$A, B$ and~$AB$, the process
 $\Phi_t = (X_t,Y_t,Z_t)$ is a continuous-time Markov chain.
 Since the states~$\Phi^A = (n, 0, 0)$ and~$\Phi^B = (0, n, 0)$ are the only absorbing states and are both accessible from all other states, it follows that with probability one,
 $$ \lim_{t \to \infty} \Phi_t \in \{\Phi^A, \Phi^B \}. $$
 The following result characterizes how long this takes, for large~$n$.
 Use $\P_{(X,Y,Z)}(\, \cdot \,)$ for the law of the process with initial configuration~$(X,Y,Z)$.
\begin{theorem}
\label{thm:finalphase}
 Let $\gamma = 1 + (-4+2\sqrt{5})^{-1}$, and define the stopping time
 $$ \begin{array}{l} T_c = \inf \,\{t : \Phi_t \in \{\Phi^A, \Phi^B \} \}. \end{array} $$
 Then, for any $\alpha>0$,
 \begin{eqnarray*}
 \lim_{n\rightarrow\infty} \ \sup_{(X,Y,Z)}\P_{(X,Y,Z)}( \, T_c/ \log n > \gamma + \alpha \, ) = 0 \\
 \lim_{n\rightarrow\infty} \ \sup_{(X,Y,Z)}\P_{(X,Y,Z)}( \, T_c/ \log n > \gamma - \alpha \, ) = 1
 \end{eqnarray*}
\end{theorem}
 Notice that, if individuals only remember the last word they heard, then starting from a configuration with two words, we obtain the voter model on the complete graph, for which the time
 to consensus is of order~$n$.
 The reason it is much faster here is because, once a majority of type~$A$ or~$B$ develops, it is maintained. To prove this result we use an ODE heuristic to get an idea of what is happening, then carve up the state space into a few pieces and use martingale estimates to control the
 behavior of sample paths on each piece.\\

The paper is laid out as follows. In Section \ref{sec:sampath} we derive a simple and useful sample path estimate for quasi-left continuous semimartingales with bounded jumps, and give some formulas that help with computations later on. This section can be read independently of the rest of the paper, and may be of use in other applications. In Section \ref{sec:early-mid} we prove Theorem \ref{thm:midphase} in several steps. In Section \ref{subsec:create} we show that about $n/2$ words are created in about $\frac{1}{2}\log n$ time, using Chebyshev's inequality and a coupon-collecting argument, respectively. In Section \ref{subsec:maintain} we show that $o(n)$ words are deleted in $n^{1/2-\ep}$ time, which as noted above is achieved by controlling the number of individuals that know a given word, and which requires the sample path estimates of Section \ref{sec:sampath}. In Section \ref{sec:final} we use an ODE comparison and the estimates of Section \ref{sec:sampath} to prove Theorem \ref{thm:finalphase}. Some additional results are collected in an Appendix, including a general sample path estimate for Poisson processes, and one for semimartingales with sublinear drift.

\section{Sample path estimation}\label{sec:sampath}

Using the semimartingale theory in \cite{jacod} we derive a useful estimate for quasi-left continuous semimartingales with bounded jumps, which can be found in Lemma \ref{lem:sm-est}. It can be thought of as a continuous-time analogue of Azuma's inequality. In this section, unless otherwise noted, references are to formulas in \cite{jacod}.\\

Given is a filtered probability space $(\Omega,\F,\bf F,P)$ satisfying the ``usual conditions'' as described in \cite{jacod}. Processes are assumed to be optional.
Given $X$, $X_{-}$ is the left continuous process $(X_{t^-})_{t \ge 0}$ and $\Delta X = X-X^-$ is the process of jumps.
$X^p$ denotes the compensator and $X^c$ the continuous martingale part, when they exist.\\

A \emph{semimartingale} $X$ is a process (on $\R$ unless specified otherwise) that can be written as $X=X_0+M+A$, where $X_0$ is an $\F_0$-measurable random variable, $M$ is a local martingale and $A$ has locally finite variation. Using I.3.17, a semimartingale is \emph{special} if it can be written as
\begin{equation}\label{eq:spec-sm}
X = X_0 + X^m + X^p 
\end{equation}
where $X^p$ is the compensator of $X$ and $X^m$ is a uniquely defined local martingale satisfying $X_0^m=0$. If $X$ is a semimartingale with \emph{bounded jumps}, that is, $|\Delta X| \le c$ for some $c>0$ then by I.4.24, $X$ is a special semimartingale and $|\Delta X^m| \le 2c$. 
If $X$ is also \emph{quasi-left continuous}, that is, $\Delta X_T=0$ a.s. on $\{T<\infty\}$, for any predictable time $T$, then using I.2.35 in the proof of I.4.24, we obtain the slightly stronger estimate $|\Delta X^m| \le c$.\\

Any (right-continuous) Markov chain with values in $\R$ is a semimartingale, since it is right-continuous and has locally finite variation, and is also quasi-left continuous, effectively because the jump times of a Poisson process are totally inaccessible; if this explanation is insufficient use Proposition 22.20 in \cite{kallenberg} and note that Markov chains are Feller processes. As shown in I.4.28, a deterministic function $f:\R_+\to \R$ is a semimartingale iff it is right-continuous with finite variation over each compact interval, and is quasi-left continuous iff it is continuous, since any fixed time is predictable.\\

We will occasionally assume $X$ is defined only up to some predictable time $\zeta$ that may be finite; in this case, information about $X$ can be recovered from the stopped processes $X^{\tau_n}$ defined by $X_t^{\tau_n} = X_{t \wedge \tau_n}$, where $\tau_n$ is an announcing sequence for $\zeta$, i.e., an increasing sequence of stopping times with limit $\zeta$.\\

If $M$ is a local martingale satisfying $M_0=0$ and $|\Delta M| \le c$ for some $c>0$, by I.4.1, $M$ is locally square-integrable, so by I.4.3, $M^2$ has a compensator, denoted $\langle M \rangle$ and called the predictable quadratic variation. Relative to the decomposition (I.4.18) $M = M^c + M^d$ into continuous and discontinuous martingale parts,
$$\langle M \rangle = \langle M^c \rangle + \langle M^d \rangle,$$
and $\langle M^d \rangle$ is the compensator of $[M^d] = \sum_{s \le t}(\Delta M_s)^2$, the quadratic variation of $M^d$.

\begin{lemma}\label{lem:exp-sm}
Let $M$ be a quasi-left continuous local martingale with $M_0=0$ and $|\Delta M| \le c$ for some $c>0$. Then,
$$\exp(M-(e^c/2)\langle M \rangle)$$
is a local supermartingale with initial value $1$.
\end{lemma}

\begin{proof}
Let $V_t$ be a continuous predictable process with locally finite variation, satisfying $V_0=0$, and let $E = \exp(M-V)$. Applying It{\^o}'s formula I.4.57 using the function $x \mapsto e^{x}$,
\begin{align}\label{eq:ito-exp}
E_t = 1 + (E_- \cdot (M - V))_t + \frac{1}{2}(E_- \cdot \langle M^c \rangle )_t  + \sum_{s \le t}E_{s^-}(e^{\Delta M_s} - 1 - \Delta M_s). 
\end{align}
Noting that $e^x - 1 - x \le \frac{1}{2}e^cx^2$ when $|x|\le c$, the last term is bounded by
$$\frac{e^c}{2}\sum_{s \le t}E_{s^-}(\Delta M_s)^2.$$
Using this bound and taking the compensator of both sides in \eqref{eq:ito-exp},
\begin{align*}
E^p & \le E_- \cdot (-V + \frac{1}{2}\langle M^c \rangle + \frac{1}{2}e^c \langle M^d \rangle) \\
& \le E_- \cdot (-V + \frac{1}{2}e^c \langle M \rangle ).
\end{align*}
The assumption of quasi-left continuity implies that $\langle M \rangle$ can be taken continuous (i.e. it has a continuous version; see I.4.3). Since it is the compensator of $[M,M]$, $\langle M \rangle$ is also predictable and has locally finite variation. Letting $V = (e^c/2)\langle M \rangle$, the same is true for $V$. With this choice of $V$, $E^p \le 0$, which implies $E$ is a local supermartingale.
\end{proof}

\begin{lemma}\label{lem:qlc}
Let $X$ be a special semimartingale with locally square-integrable martingale part $X^m$. Then, $X$ is quasi-left continuous iff $X^p$ and $\langle X^m \rangle$ are continuous.
\end{lemma}

\begin{proof}
If $X$ is quasi-left continuous (qlc), then by I.2.35 its predictable projection $^p X = X_-$. Since the $^p( \, \cdot \,)$ operation is linear (which follows from uniqueness and property (ii) in I.2.28), and since $^p X_- = X_-$, it follows that
$$^p(\Delta X) = \,^p X - \,^p X_- = 0.$$
Since $X$ is special it has a compensator $X^p$, and by I.3.21, $\Delta(X^p) = \,^p(\Delta X)$. By the above, this is $0$, i.e., $X^p$ is continuous. Using \eqref{eq:spec-sm}, $\Delta X^m = \Delta X$ which implies that $X^m$ is qlc, by definition of qlc. Since $X^m$ is locally square-integrable, by I.4.3, $\langle X^m \rangle$ is continuous.\\

On the other hand, if $\langle X^m \rangle$ is continuous then by I.4.3, $X^m$ is qlc. If in addition $X^p$ is continuous then $\Delta X = \Delta X^m$ which implies $X$ is qlc, by definition of qlc.
\end{proof}

\begin{lemma}\label{lem:sm-est}
Let $X$ be a quasi-left continuous semimartingale such that $|\Delta X| \le c$ for some $c>0$. Then for $\lambda,a>0$ and $\bullet \in \pm$,
\begin{equation}\label{eq:sm-est}
P( \bullet (X_t - X_0 - X_t^p) \ge a +  (\lambda e^{\lambda c}/2) \langle X^m \rangle_t ) \le e^{-\lambda a}
\end{equation}
\end{lemma}

\begin{proof}
Notice that $X_t - X_0 - X_t^p=X_t^m$ and that for $\lambda>0$ and $\bullet \in \pm$, $\langle \bullet \lambda X^m \rangle = \lambda^2\langle X^m\rangle$. As noted just above \eqref{eq:spec-sm}, since $X$ has bounded jumps it is special and by Lemma \ref{lem:qlc}, $X^m$ is qlc. Take $M = \bullet \lambda X^m$ in Lemma \ref{lem:exp-sm}, which has $|\Delta M| \le \lambda c$, and use Doob's inequality to find
$$P(\bullet\lambda X_t^m - (\lambda^2 e^{\lambda c}/2 )\langle X^m \rangle_t \ge \lambda a) \le e^{-\lambda a}.$$
\end{proof}

For practicality's sake we'll use a slightly cruder version of \eqref{eq:sm-est}. Since $1/2 \le \log 2$, if $\lambda c \le 1/2$ then $e^{\lambda c} \le 2$, so from \eqref{eq:sm-est} it follows that for $a>0$ and $\bullet \in \pm$, 
\begin{equation}\label{eq:sm-est2}
\text{if} \quad 0<\lambda c \le 1/2 \quad \text{then} \quad P( \bullet (X_t - X_0 - X_t^p) \ge a +  \lambda \langle X^m \rangle_t ) \le e^{-\lambda a}.
\end{equation}

Using Lemma \ref{lem:qlc} as inspiration, say that a special semimartingale $X$ with locally square-integrable martingale part $X^m$ is \emph{quasi-absolutely continuous} (qac) if both $X^p$ and $\langle X^m \rangle$ are absolutely continuous. In this case define the \emph{drift} $\mu(X) = (\mu_t(X))_t$ and the \emph{diffusivity} $\sigma^2(X) = (\sigma_t(X))_t$ for Lebesgue-a.e. $t$ by
\begin{equation}\label{eq:mu-sig}
\mu_t(X) = \frac{d}{dt}X^p_t, \quad \sigma^2_t(X) = \frac{d}{dt}\langle X^m \rangle_t.
\end{equation}
\indent For deterministic processes, qac is equivalent to absolute continuity, since $\mu_t(f) = f(t)$, $\sigma^2_t(f) = 0$ and absolute continuity implies locally finite variation. For Markov chains $X$ on $\R$ with jump measure $\alpha(x,dy)$, if qac holds then $\mu$ and $\sigma$ are given by functions
$$\mu(x) = \int_{\R} y\alpha(x,dy), \quad \sigma^2(x) = \int_{\R} y^2\alpha(x,dy),$$
i.e., $\mu_t(X) = \mu(X_t)$ and $\sigma^2_t(X) = \sigma^2(X_t)$. Conversely, if $|\Delta X| \le c$ and the total intensity $q(x) = \int_{\R}\alpha(x,dy)$ of the jump measure is bounded on compact subintervals of $\R$, then $X$ is qac up to the first explosion time $\sup_{r>0}\inf\{t:|X_t|\ge r\}$, and $\sigma^2(x)\le c^2q(x)$.

\begin{lemma}\label{lem:qac-prod}[Product rule]
Suppose $X_t,Y_t$ are qac semimartingales on a common filtered probability space. Then $\lng X^m,Y^m \rng$ exists and is absolutely continuous, and $(XY)^p$ exists, is absolutely continuous, and $\mu_t(XY) = \frac{d}{dt}(XY)^p_t$ is given by
$$\mu(XY) = \sigma(X,Y) + X_- \mu(Y) + Y_- \mu(X),$$ 
where $\sigma_t(X,Y) = \frac{d}{dt} \lng X^m,Y^m\rng_t$.
\end{lemma}

\begin{proof}
By definition of quadratic variation,
$$XY = X_0Y_0 + [X,Y] + X_- \cdot Y + Y_- \cdot X.$$
Since $X= X_0+X^p+X^m,\, Y=Y_0+Y^p+Y^m$ and $X_0+X^p, \, Y_0+Y^p$ have locally finite variation, $[X,Y] = [X^m,Y^m]$. Since $X^m, \, Y^m$ are locally square-int, $[X^m,Y^m]$ has compensator $\langle X^m,Y^m \rangle$, so $XY$ has compensator
$$(XY)^p = \lng X^m, Y^m \rng + X_- \cdot Y^p + Y_- \cdot X^p.$$
The result will follow if we can show $\lng X^m,Y^m \rng$ is absolutely continuous. For any $s<t$, applying the Cauchy-Schwarz inequality to the symmetric, bilinear and semidefinite map $(X,y) \mapsto \langle X,Y\rangle_t - \langle X,Y \rangle_s$ gives
$$|\langle X^m,Y^m \rangle_t - \langle X^m, Y^m \rangle_s| \le \sqrt{ (\langle X^m \rangle_t - \langle X^m \rangle_s)(\langle Y^m \rangle_t - \langle Y^m \rangle_s)}.$$
Absolutely continuity of $t\mapsto \lng X^m \rng_t, \, \lng Y^m \rng_t$ means that for any $\eps>0$ there is $\dlt>0$ so that if $\sum_i|t_i-s_i| < \delta$ then $\sum_i|\lng X^m \rng_{t_i} - \lng X^m \rng_{s_i}|, \, \sum_i|\lng Y^m \rng_{t_i} - \lng Y^m \rng_{s_i}|<\eps$. Using the Cauchy-Schwarz inequality to obtain the second line,
\begin{align*}
\sum_i|\langle X^m,Y^m \rangle_{t_i} - \langle X^m, Y^m \rangle_{s_i}| & \le \sum_i \sqrt{(\lng X^m \rng_{t_i} - \lng X^m \rng_{s_i})(\lng Y^m \rng_{t_i} - \lng Y^m \rng_{s_i})} \\
& \le \left(\sum_i |\lng X^m \rng_{t_i}-\lng X^m \rng_{s_i}| \sum_i |\lng Y^m \rng_{t_i} - \lng Y^m \rng_{s_i}| \right)^{1/2} \\
& < (\eps \cdot \eps)^{1/2} = \eps
\end{align*}
which shows that $\lng X^m,Y^m \rng$ is absolutely continuous.
\end{proof}

\section{Early and middle phases}\label{sec:early-mid}

In this section we consider the behaviour of $|V_t|$ for $t\le n^{1/2-o(1)}$. Define
$$V_t^o = \bigcup_{(v,s):s \le t} W_s(v)\quad \hbox{and} \quad V_t^{\times} = V_t \setminus V_t^o,$$
respectively the number of words created up to time $t$, and the number of words created and then deleted by time $t$. Theorem \ref{thm:midphase} is implied by the following two propositions, whose proof is the objective of this section.\\

\begin{proposition}\label{prop:early}
For each $\ep>0$, $\lim_{n\to\infty}\P(\sup_{t \ge (\frac{1}{2}+\ep)\log n}| \ |V_t^o| - \frac{n}{2}| \ge n^{1/2+\ep}) = 0.$
\end{proposition}
\begin{proposition}\label{prop:middle}
For each $\ep>0$, $\lim_{n\to\infty}\P(\sup_{t \le n^{1/2-\ep}}|V_t^\times| = o(n))=1.$
\end{proposition}
In words, in order to estimate $|V_t|$ we obtain good control on $|V_t^o|$, then show that $|V_t^\times|$ is not too big. We begin with $V_t^o$.

\subsection{Creation of vocabulary}\label{subsec:create}
Our first task is to prove Proposition \ref{prop:early}, and to do so we show that $|V_t^o|$ rises from $0$ to $n/2 + O(n^{1/2+o(1)})$ within $\frac{1}{2}\log n$ time, then remains constant.
For a vertex $v$ let $N_t(v) = |W_t(v)|$ denote the size of the vocabulary of individual $v$, and let 
$$T_o = \inf\{t: \min_v N_t(v) \ge 1\}$$
be the first time that every individual knows at least one word. Clearly $V_t^o$ is non-decreasing as a set, so $V_{\infty}^o = \lim_{t\to\infty}V_t^o$ exists and $|V_{\infty}^o| \le n$. Once everyone knows a word, no new words are created, so $V_t = V_{T_o}^o=V_{\infty}^o$ for $t\ge T_o$. Proposition \ref{prop:early} is implied by the following two lemmas, in which we estimate $T_o$ and $V_{T_o}^o$.

\begin{lemma}\label{lem:early1}
For $c \ge 0$,
$$\P(|T_o -\frac{1}{2}\log n| \ge c) \le 2e^{-c} + o(1) \ \hbox{as} \ n \to \infty.$$
\end{lemma}

\begin{lemma}\label{lem:early2}
 Let~$X=|V_{T_o}^o|$ be the number of words ever created. Then,
 $$ \begin{array}{l} \lim_{n \to \infty} P (|X - n/2| \geq n^{\alpha}) = 0 \quad \hbox{for all} \quad \alpha > 1/2. \end{array} $$
\end{lemma}

\begin{proof}[Proof of Lemma \ref{lem:early1}]
Let $M_t = \{v:N_t(v)=0\}$ denote \emph{mute} vertices, those not yet knowing a word, and observe that $T_o\le t$ is equivalent to $|M_t|=0$. For each distinct ordered pair of vertices $(v,w)$, at rate $(n-1)^{-1}$, the directed edge $(v,w)$ has an event, and both $v$ and $w$ are removed from $M_t$, if either or both still belongs. If we let $Z_t = |M_t|$ denote the number of mute vertices at time $t$, it follows that $Z_t$ is a Markov chain with $Z_0=n$ and transitions
$$Z_t \to \begin{cases} Z_t-1 & \hbox{at rate} \ 2(n-1)^{-1}Z_t(n-Z_t), \ \hbox{and} \\
Z_t-2 & \hbox{at rate} \ (n-1)^{-1}Z_t(Z_t-1).\end{cases}$$
We find that
\begin{eqnarray*}
\lim_{h\to 0^+}h^{-1}\E[Z_{t+h}-Z_t \ \mid \ Z_t=z] &=& -2(n-1)^{-1}z(n-z) - 2(n-1)^{-1}z(z-1)\\
&=& -2(n-1)^{-1}(nz - z^2 + z^2 - z) \\
&=& -2(n-1)^{-1}(n-1)z = -2z.
\end{eqnarray*}
Letting $m(t) = \E[Z_t]$, $m(0)=n$ and taking expectations in the above, $m'(t) = -2m(t)$, which has the unique solution $m(t) = ne^{-2t}$. Fix $c \in \R$ and let $t_c=\frac{1}{2}\log n + c$. Using Markov's inequality,
$$\P(T_o > t_c) = \P(Z_{t_c}\ge 1) \le \E[Z_{t_c}] = e^{-2c}.$$
To get a lower bound we turn to $Z_t^2$, which has transitions
$$Z_t^2 \to \begin{cases} Z_t^2-Z_t+1 & \hbox{at rate} \ 2(n-1)^{-1}Z_t(n-Z_t), \ \hbox{and} \\
Z_t^2-4Z_t+4 & \hbox{at rate} \ (n-1)^{-1}Z_t(Z_t-1),\end{cases}$$
so
\begin{eqnarray*}
\lim_{h\to 0^+}h^{-1}\E[Z_{t+h}^2-Z_t^2 \ \mid \ Z_t=z] &=& -(2z-1)2(n-1)^{-1}z(n-z) - (4z-4)(n-1)^{-1}z(z-1)\\
&=& -4z(n-1)^{-1}((z-\frac{1}{2})(n-z) + (z-1)^2) \\
&=& -4z(n-1)^{-1}(nz - z^2 - \frac{n}{2} + \frac{z}{2} + z^2 - 2z + 1) \\
&=& -4z(n-1)^{-1}((n-\frac{3}{2})z + 1 - \frac{n}{2}) \\
&=& \frac{2(n-2)}{n-1}z - \frac{4(n-3/2)}{n-1}z^2.
\end{eqnarray*}
Letting $\nu(t) = \E[Z_t^2]$, $\nu(0) = n^2$ and taking expectations above,
$$\nu'(t) = - 4(1 - (2(n-1))^{-1})\nu(t)^2 + 2(1 - (n-1)^{-1})m(t),$$
so letting $\gamma = 4 - 2/(n-1)$, using $m(t)=ne^{-2t}$ and solving the above DE, we find
\begin{eqnarray*}
\nu(t) &=& n^2e^{-\gamma t} + 2(1-1/(n-1))ne^{-\gamma t}(e^{(\gamma-2)t} - 1)/(\gamma -2).
\end{eqnarray*}
As above let $t_c = \frac{1}{2}\log n + c$, then $m(t_c) = e^{-2c}$ and for fixed $c$,
$$\nu(t_c) = e^{-4c} + e^{-2c} + o(1) \ \hbox{as} \ n\to\infty,$$
so $\var(Z_{t_c}) = \nu(t_c) - m(t_c)^2 = e^{-2c} + o(1)$. Using Chebyshev's inequality,
$$\P(T_o \le t_c) = \P(Z_t = 0) \le \P(|Z_t-\E[Z_t]| \ge \E[Z_t]) \le \var(Z_{t_c})/\E[Z_{t_c}]^2 \le \frac{e^{-2c}+o(1)}{e^{-4c}} = e^{2c} + o(1).$$
The result follows by taking a union bound of both estimates.
\end{proof}
We note in passing that $|V_0^o|=0$ and $|V_t^o|$ increases by $1$ at rate $Z_t$. Heuristically, $Z_t \approx ne^{-2t}$, so $|V_t^o| \approx (n/2)(1-e^{-2t})$, for $t \le \frac{1}{2}\log n$. This can be made precise using stochastic calculus, although we do not pursue it here.

\begin{proof}[Proof of Lemma \ref{lem:early2}]
 Letting~$X_v$ for each vertex~$v \in V$ be the Bernoulli random variable equal to one if and only if~$v$ speaks before listening, by construction
 and obvious symmetry, we have
 $$ \begin{array}{l} X = \sum_{v \in V} X_v \quad \hbox{and} \quad P (X_v = 0) = P (X_v = 1) = 1/2. \end{array} $$
 It follows that the expected number of words is given by
\begin{equation}
\label{eq:mean}
  \begin{array}{l} E (X) = \sum_{v \in V} E (X_v) = \sum_{v \in V} P (X_v = 1) = n/2. \end{array}
\end{equation}
 To also compute the variance, fix~$v, w \in V$ and let~$B$ be the event that the first edge becoming active starting from~$v$ or~$w$ is edge~$vw$.
 Since there are~$n - 1$ edges starting from each vertex,
\begin{equation}
\label{eq:variance-1}
  P (B) = \frac{1}{2 (n - 1) - 1} = \frac{1}{2n - 3}.
\end{equation}
 In addition, the two vertices cannot both speak before listening when~$B$ occurs whereas the two events are independent on the event~$B^c$ therefore
\begin{equation}
\label{eq:variance-2}
  \begin{array}{rcl}
    P (X_v = X_w = 1 \,| \,B) & = & 0 \vspace*{4pt} \\
    P \,(X_v = X_w = 1 \,| \,B^c) & = & P (X_v = 1\,| \,B^c) \,P (X_w = 1\,| \,B^c) = 1/4. \end{array}
\end{equation}
 Combining~\eqref{eq:variance-1}--\eqref{eq:variance-2}, we deduce that
 $$ \begin{array}{rcl}
      E (X^2) & = & \displaystyle \sum_{v \in V} P (X_v^2 = 1) + \sum_{v \neq w} P (X_v = X_w = 1) \vspace*{4pt} \\
              & = & \displaystyle \sum_{v \in V} \ \frac{1}{2} + \sum_{v \neq w} \ \frac{1}{4} \ \frac{2n - 4}{2n - 3}
                =   \frac{n}{2} \,\bigg(1 + \frac{(n - 1)(n - 2)}{2n - 3} \bigg) \end{array} $$
 which, together with some basic algebra, gives the variance
\begin{equation}
\label{eq:variance-3}
  \var (X) = \frac{n}{2} \,\bigg(1 + \frac{(n - 1)(n - 2)}{2n - 3} - \frac{n}{2} \bigg) = \frac{n}{4} \,\bigg(\frac{n - 2}{2n - 3} \bigg) = O (n).
\end{equation}
 From~\eqref{eq:mean} and~\eqref{eq:variance-3} and Chebyshev's inequality, we conclude that
 $$ \begin{array}{l} \lim_{n \to \infty} P (|X - n/2| \geq n^{\alpha}) \leq \lim_{n \to \infty} n^{- 2 \alpha} \var (X) = 0 \end{array} $$
 for all~$\alpha > 1/2$. This completes the proof.
\end{proof}

\subsection{Maintenance of vocabulary}\label{subsec:maintain}
Next, we prove Proposition \ref{prop:middle}, that says that with probability tending to $1$ as $n\to\infty$,
$$\sup_{t \le n^{1/2-o(1)}} |V_t^{\times}| = o(n).$$
Clearly $V_t^\times$, like $V_t^o$, is non-decreasing, since once a word vanishes from the population, it does not come back.
 We first bound $|V_t^\times|$ by a simpler quantity.
 Say that \emph{agreement} upon word $y$ occurs at $(v,w,t)$ if
 $$y \in W_{t^-}(w) \ \hbox{and} \ v \ \hbox{speaks word} \  y \ \hbox{to} \ w \ \hbox{at time} \ t.$$
 If word $w$ is created at some time $s \le t$, then $w \in W_s(w)$, and remains in individual $w$'s vocabulary at least until the first time $t>s$ that agreement occurs at $(\cdot,w,t) \ \hbox{or} \ (w,\cdot,t)\}$. This implies
 $$V_t^\times \subseteq H_t = \{w:\hbox{agreement occurs at} \ (\cdot,w,s) \ \hbox{or} \ (w,\cdot,s) \ \hbox{for some} \ s \le t\}.$$
In words, in order to delete a word $w$ from the population, it must at least be deleted from its source. Since each agreement contributes at most 2 to $H_t$, it follows that
$$\begin{array}{rl}
|V_t^\times| \le 2A_t\quad \hbox{where}\quad A_t =& |\{s \le t: \ \hbox{agreement occurs at} \ (\cdot,\cdot,s)\}|\\
& \ (\hbox{number of agreements up to time} \ t).
\end{array}$$
 In order to control $A_t$ we first define some useful observable quantities. For $w \in V$ we recall the \emph{cluster} $\C_t(w)$ of $w$, that is, the set of individuals that know word $w$ at time $t$:
$$\C_t(w) = \{v:w \in W_t(v)\}.$$
Recall that $N_t(v) = |W_t(v)|$ denotes the size of the vocabulary of individual $v$, and let
\begin{equation}\label{eq:rw}
R_t(w) = \1(N_t(w)=0) + \sum_{v \in \C_t(w)}1/N_t(v)
\end{equation}
denote the rate at which word $w$ is spoken. Let $J(w,v)$ denote the times at which $w$ speaks to $v$, and let
$$N_t^{\ell}(v) = \sum_w |J(w,v) \cap [0,t]| = \hbox{number of listening events for} \ v \ \hbox{up to time} \ t,$$
noting that $N_t(v) \le N_t^{\ell}(v)$ and $\{(N_t^{\ell}(v)):v \in V\}$ is a collection of independent Poisson processes with intensity 1. If we let
$$\begin{array}{rcl}
\tau_a(v) &=& \inf\{t: v \in H_t\} \quad \hbox{and} \\
\tau_a(v,t) &=& 0 \vee \sup\{s \le t:\hbox{agreement occurs at} \ (v,\cdot,s) \ \hbox{or} \ (\cdot,v,s)\},
\end{array}$$
then $N_t(v) = N_t^{\ell}(v) - N_{\tau_a(v,t)^-}^{\ell}(v)$, and in particular,
$$N_t(v) = N_t^{\ell}(v) \ \hbox{for} \ t < \tau_a(v).$$\\

Let $S_t(w) = |\C_t(w)|$ and $P_t(w) = (S_t(w)-1)/(n-1)$, and let $S_t = \max_w S_t(w)$. Each site $v$ that knows word $w$ speaks it at rate $N_t(v)^{-1}/(n-1)$ to each of the other $S_t(w)-1$ sites in $\C_t(w)$. Letting
$$A_t(w) = |\{s \le t: \hbox{agreement occurs upon word} \ w \ \hbox{at time} \ s\}|,$$
so that $A_t = \sum_w A_t(w)$, it follows that $A_t(w)$ increases by $1$ at rate
$$(S_t(w)-1)\sum_{v \in \C_t(w)}\frac{N_t(v)^{-1}}{n-1} = R_t(w)P_t(w).$$
Since $\sum_{w \in V}R_t(w) = n$ is the total speaking rate and $P_t(w) \le (S_t-1)/(n-1) \le S_t/n$, summing the above display over $w \in V$ we find
\begin{equation}\label{eq:agree-rate}
A_t \ \hbox{increases by} \ 1 \ \hbox{at rate at most} \ S_t.
\end{equation}

We have reduced the problem of controlling $|V_t^\times|$ to that of controlling $S_t$. The following becomes the goal of this subsection. Since its proof has a few parts, we call it a theorem.

\begin{theorem}\label{thm:middle}
For small $\ep>0$,
$$\lim_{n\to\infty}\P(\sup_{t \le n^{1/2-\ep}} \frac{S_t}{(1+t)^{1+\ep}} \ge (\log n)^9) = 0.$$
\end{theorem}

Before moving onto the proof of Theorem \ref{thm:middle} we first use it to obtain Proposition \ref{prop:middle}.

\begin{proof}[Proof of Proposition \ref{prop:middle}]
From \eqref{eq:agree-rate}, for any $T>0$, $\sup_{t \le T}A_t \le \poi( \int_0^T S_udu)$. Using Theorem \ref{thm:middle}, with probability $1-o(1)$ as $n\to\infty$
$$\bay{rcl}
\int_0^{n^{1/2-\ep}} S_udu &\le & (\log n)^9\int_0^{n^{1/2-\ep}} (1+u)^{1+\ep}du \le (\log n)^9(2+\ep)^{-1}(1+n^{1/2-\ep})^{2+\ep} \\
&=& O((\log n)^9n^{1 - 3\ep/2 - \ep^2}) = o(n).
\eay$$
Since $\P(\poi(\lambda) \le 2\lambda) \to 1$ as $\lambda\to\infty$ it follows that $\sup_{t \le n^{1/2-\ep}}A_t = o(n)$ with probability $1-o(1)$, and since $|V_t^\times| \le 2A_t$, the same is true for $|V_t^\times|$.
\end{proof}

To begin the proof of Theorem \ref{thm:middle} we introduce a modified construction to help us make a coupling. First, for each ordered triple $(y,z,v)$ let $R_t(y,z,v)$ be the rate at which word $y$ is spoken by site $z$ to $v$, let $R_t(y,v) = \sum_z R_t(y,z,v) $ be the rate at which site $v$ hears word $y$, and as above let $R_t(y)=\sum_v R_t(y,v)$ be the rate at which word $y$ is spoken. We calculate
\begin{equation}\label{eq:r1}
\begin{array}{rcl}
R_t(y,z,v) &=& (N_t(z))^{-1}\1(z \in \C_t(y), \ z \ne v) + \1(y=z \ne v,\,N_t(y)=0))/(n-1) \ \hbox{and}\\
R_t(y,v) &=& (\1(N_t(y)=0,v \ne y) + \sum_{z \in \C_t(y) \setminus \{v\}}N_t(z)^{-1} )/(n-1). \\
\end{array}
\end{equation}
Clearly $\sum_y R_t(y,v)=1$ for each $v,w$ and $t\ge 0$. Fix an ordering $v_1<\dots<v_n$ of $V$ and define an independent family $\{U_v:v \in V\}$ of augmented Poisson point processes with intensity 1, that will correspond to listening events. For $v \in V$, $1 \le i,j  \le n$ and $t\ge 0$ let
$$I_t(v,i,j) = \left[\sum_{k=1}^{i-1}R_t(v_k,v) + \sum_{m=1}^{j-1}R_t(v_i,v_m,v),\sum_{k=1}^{i-1}R_t(v_k,v) + \sum_{m=1}^jR_t(v_i,v_m,v)\right),$$
noting that $\{I_t(v,i,j):1 \le i,j\le n\}$ partitions $[0,1)$.
 Then, if $(t,u)\in U_v$ and $u\in I_{t^-}(v,i,j)$, word $v_i$ is spoken by $v_j$ to $v$, which defines the process. Using this construction and given $C,R>0$ we obtain upper bounds $\Cu_t(w),\Ru_t(w)$ on $\C_t(w),R_t(w)$ for all $w \in V$, valid up to the time
$$\begin{array}{rcl}
T_{C,R} &=& \min_{w \in V}T_w(C,R) \quad \hbox{where} \\
T_w(C,R) &=& \inf\{t: \sum_{v \in \Cu_t(w) \cap H_t}N_t(v)^{-1} \ge C \ \hbox{or} \ \Ru_t\ge R(w)\}.
\end{array}$$
That is, we obtain for each $w \in V$ a pair of processes $\Cu_t(w), \Ru_t(w)$ with nice properties, such that $\C_t(w) \subseteq \Cu_t(w)$ and $R_t(w) \le \Ru_t(w)$ for $t \le T_{C,R}$ pointwise on realizations of the process.
Given $w \in V$, $\Cu_t(w),\Ru_t(w)$ are non-decreasing and defined as follows. For $i \in \{1,\dots,n\}$ let
$$\begin{array}{rcl}
b_t(v,i) &=& \sum_{k=1}^{i-1}R_t(v_k,v),\ \hbox{and for} \ x\in [0,1) \ \hbox{let} \\
I_t(v,i,x) &=& [b_t(v,i), \ b_t(v,i) + x) \mod 1.
\end{array}$$
Define
$$N_t^{\ell}(v,i,R) = |\{(s,u) \in U_v:s \le t, \ u \notin I_t(v,i,R)\}| \le N_t^{\ell}(v).$$
Let $i$ be such that $w=v_i$. Initially, $\Cu_0(w)=\{w\}$ and $\Ru_0(w)=1+C$. $\Ru_t(w)$ is defined as follows.
$$\Ru_t(w) = 1+C+\sum_{v \in \Cu_t(w) \setminus\{w\}}1/(1+N_t^{\ell}(v,i,R)).$$
Then, $\Cu_t(w)$ is defined as follows.
$$\begin{array}{rll}
&\hbox{if}& \ (t,u) \in U_v \ \hbox{and} \ u \in I_{t^-}(v,i,\Ru_{t^-}/(n-1)), \\
&\hbox{then}& \Cu_t = \Cu_{t^-}\cup \{v\}.
\end{array}$$
We demonstrate the claimed comparison.

\begin{lemma}\label{lem:C-comp}
For each $w\in V$, $\C_t(w)\subseteq \Cu_t(w)$ and $R_t(w) \le \Ru_t(w)$ for $t < T_{C,R}$.
\end{lemma}
\begin{proof}
Let
$$\tau_c(w) = \inf\{t:\C_t(w) \ne \varnothing\},$$
then $\C_t(w) \subset \{w\} \subseteq \Cu_t(w)$ and $R_t(w) \le 1 \le \Ru_t$ for $t < \tau_c(w)$.
 For the remainder, assume $t\ge \tau_c(w)$ and let $i$ be such that $w=v_i$.
 By construction, $v \in V$ is added to $\C_t(w)$ if
\begin{equation}\label{eq:C-up-trans}
v \notin \C_{t^-}(w), \ (t,u) \in U_v \ \hbox{and} \ u \in I_{t^-}(v,i,R_t(w,v))
\end{equation}
and otherwise, $\C_t(w)$ does not increase. If $t \ge \tau_c(w)$ then $N_t(w)\ge 1$, and if $z \notin H_t$ then $N_t^{\ell}(z) = N_t(z)$. So, from the second line of \eqref{eq:r1},
$$\begin{array}{rcl}
(n-1)R_t(w,v) &=& \sum_{z \in \C_t(w)\setminus\{v\}}1/N_t(z) \\
&\le & \sum_{z \in \C_t(w)}1/N_t^{\ell}(z)+\sum_{z \in \C_t(w) \cap H_t}1/N_t(z).
\end{array}$$
If $w \in \C_t(w)$ then $N_t(w)^{-1} \le 1$. By definition of $T_{C,R}$, if $\C_t(w) \subseteq \Cu_t(w)$ and $t < T_{C,R}$ then
$$(n-1)R_t(w,v) \le (1 + C + \sum_{z \in \C_t(w)\setminus\{w\}} 1/N_t^{\ell}(z)).$$
If $v \in \Cu_t(w)$ and $t < T_{C,R}$ then $N_t^{\ell}(v) \ge N_t^{\ell}(v,R)+1$, since this implies existence of a point in
$$U_v \cap \{(s,u):s \le t \ \hbox{and} \ u \in I_s(v,i,\Ru_t(w)/(n-1))\},$$
which is counted in $N_t^{\ell}(v)$ but not in $N_t^{\ell}(v,R)$.
 If $\C_t(w) \subseteq \Cu_t(w)$ it follows that $R_t(w,v) \le \Ru_t(w)/(n-1)$ for each $v$ which implies the containment $\C_t(w) \subseteq \Cu_t(w)$ is preserved across transitions \eqref{eq:C-up-trans} that cause $\C_t(w)$ to increase.
 Since $\Cu_t(w)$ is non-decreasing and transitions are well-ordered this implies $\C_t(w) \subseteq \Cu_t(w)$ for $t < T_{C,R}$.
 It remains to check $R_t(w) \le \Ru_t(w)$ for $\tau_c(w)\le t<T_{C,R}$. But in this case, \eqref{eq:rw} and the previous argument give
 $$R_t(w) = \sum_{v \in \C_t(w)}1/N_t(v) \le 1 + C + \sum_{v \in \C_t(w)\setminus \{w\}}1/N_t^{\ell}(v) \le \Ru_t(w).$$
 
\end{proof}
Next we fix $w$ and examine $\Cu_t(w),\Ru_t(w)$ assuming $t<T_{C,R}$, and dropping the $(w)$ for neatness. Notice that $|\Cu_t|$ is non-decreasing and increases by $1$ at rate at least $(1+C)(n-|\Cu_t|)/(n-1)$, which implies $\lim_{t\to\infty}|\Cu_t|=n$. Since $|\Cu_t|$ increases by one at a time, let $y_1,\dots,y_n$ be the order in which vertices are added to $\Cu_t$, with $w=y_1$, and condition on $(y_1,\dots,y_n)$. We track $Z_t = |\Cu_t|$ and $N_t^i = N_t^{\ell}(y_i,R), \ i=1,\dots,n$ which suffices to determine $\Cu_t,\Ru_t$. Let $t_i = \inf\{Z_t=i\}$ denote the time at which $y_i$ is added to $\Cu_t$, and let $k$ be such that $w=v_k$. For $i \in \{2,\dots,n\}$, $t_i$ is the least value of $t$ such that there is a point
$$(t,u) \in \bigcup_{j \ge i}U_{y_j} \cap \{(s,v):s \in [t_{i-1},\infty), \ v \in I_s(y_j,k,\Ru_{t_{i-1}}/(n-1))\},$$
and in addition, this point belongs to $U_{y_i}$. Using this and basic properties of exponential random variables, together with the thinning property of the Poisson process, we find that conditioned on $(y_1,\dots,y_n)$,
$$(Z_t,N_t^1,\dots,N_t^n)_{t < T_{C,R}}$$
is a Markov chain with the following transitions:
$$\begin{array}{rcl}
Z_t \to Z_t+1 \ &\hbox{at rate}& \ \Ru_t(n-Z_t)/(n-1), \ \hbox{and}\\
\hbox{for} \ i=1,\dots,n, \ N_t^i \to N_t^i + 1 \ &\hbox{at rate}& \ 1 - R/(n-1).
\end{array}$$
In particular, $\{(N_t^i)_{t < T_{C,R}}:i=1,\dots,n\}$ is an i.i.d. collection of Poisson processes with intensity $1-R/(n-1)$. Since the above does not depend on the choice of values for $(y_1,\dots,y_n)$ the same holds unconditionally. Thus $Z_t$ can be viewed as follows: initially $Z_0=1$, then subject to the random environment determined by the $\{(N_t^i)\}_{i=2}^n$, $Z_t$ increases by $1$ at rate $\Ru_t(n-Z_t)/(n-1)$. Let
$$\Lambda_t(z) = 1+ C + \sum_{i=2}^{z}1/(1+N_t^i)$$
and let $(X_t)$ denote the process with $X_0=1$ that increases by $1$ at rate $\Lambda_t(X_t)$. Since $(n-Z_t)/(n-1)\le 1$ and $\Lambda_t$ is non-decreasing in $z$, it follows that
\begin{equation}\label{eq:Z-dom}
(Z_t,\Ru_t)_{t < T_{C,R}} \quad \hbox{is dominated by} \quad (X_t,\Lambda_t(X_t)).
\end{equation}
We can think of $(X_t)$ as a branching process with immigration rate $1+C$, in which individual $i$ produces offspring at the time-decreasing rate $1/(1+N_t^i)$. Two tasks lie ahead. The first is to estimate $(X_t)$. The second is to estimate $T_{C,R}$. We then combine the results to obtain Theorem \ref{thm:middle}. This is outlined as follows.

\begin{proposition}\label{prop:max-cluster}
Let $b=1+C$. For small $\ep>0$, $b \le (8 \log n)^4$ and $R=o(n)$,
$$\lim_{n\to\infty}\P(\sup_{t \le T_{C,R}}S_t/(1+t)^{1+\ep} > (\log n)^9) = 0.$$
\end{proposition}

\begin{proposition}\label{prop:TCR}
For small $\ep>0$, $b = (8\log n)^4$ and $R = b + (\log n)^{11}$,
$$\lim_{n\to\infty}\P(T_{C,R} \le n^{1/2-\ep}) = 0.$$
\end{proposition}

\begin{proof}[Proof of Theorem \ref{thm:middle}]
Use Propositions \ref{prop:max-cluster} and \ref{prop:TCR} with $b=(8\log n)^4$ and $R=b+(\log n)^{11}$.
\end{proof}

\subsubsection{Estimation of $(X_t)$}
Since $n$ does not appear in the definition of $(X_t)$ we may as well define it using an infinite sequence $\{(N_t^i)_{t \ge 0}:i=1,2,\dots\}$ of Poisson processes with intensity $r = 1-R/(n-1)$. Clearly $r\le 1$. Since $R$ will be chosen $o(n)$, we will have $r \to 1$ as $n\to\infty$, so throughout we assume $r\ge 1/2$.\\

We begin with a useful heuristic. Let $b=1+C$. Replacing $N_t^i$ with its expectation $rt$, $X_t$ increases by $1$ at rate $b+X_t(1+rt)^{-1}$, which we approximate with the differential equation
$$x' = b+x/(1+rt).$$
Let $m(t) = \exp(\int_0^t (1+rs)^{-1}ds) = (1+rt)^{1/r}$. The above equation is linear and has solution
$$x(t) = m(t)x(0) +bm(t)\int_0^tds/m(s).$$
If $r$ is close to $1$ then $x(t)$ grows just a bit faster than linearly in time. In order to analyze $(X_t)$ we break it up into two steps:
\begin{enumerate}[noitemsep]
\item Up to a fixed time $T$, when the $N_t^i$ are fairly small.
\item From time $T$ to $\infty$, when the $N_t^i$ are fairly large.
\end{enumerate}
The reason to do this is because the estimates that say $|N_t^i - rt| = o(rt)$ are only effective once $rt$ has had time to increase. The following is the main result of this subsection.

\begin{proposition}\label{prop:branch-estimate}
Let $b=1+C$. There exist $M,x_0 \in [1,\infty)$ so that for $r \ge 1/2$ and $x \ge b\vee x_0$,
$$\P(\sup_{t \ge 0} X_t - Mx(x + \log(1+t))(1+t)^{1/r} > 0) \le 19x^{3/4}e^{-x^{1/4}/4}.$$
\end{proposition}

\noindent Recall $S_t = \max_w |\C_t(w)|$. Using this result we can prove Proposition \ref{prop:max-cluster}.

\begin{proof}[Proof of Proposition \ref{prop:max-cluster}]
For each $w \in V$, using Lemma \ref{lem:C-comp} and \eqref{eq:Z-dom},
$$(|\C_t(w)|)_{t \le T_{C,R}} \quad \hbox{is dominated by} \quad (X_t).$$
Applying the result of Proposition \ref{prop:branch-estimate} and taking a union bound over $w$, if $r \ge 1/2$ and $x \ge x_0$ then
$$\bay{rl}
&\P(\sup_{t \le T_{C,R}} S_t - \Phi(t,x)>0) \le 19nx^{3/4}e^{-x^{1/4}/4}, \ \hbox{where} \\
&\Phi(t,x) = Mx(x+ \log(1+t))(1+t)^{1/r}
\eay$$
If $R=o(n)$ then recalling that $r=1-R/(n-1)$, $1/r \le 1+\ep/3$ for large $n$. Since $r\le 1$, $1+rt \le 1+t$, and if $\ep>0$ is small then $r \ge 1/2$. Letting $x = (8\log n)^4$, the probability is $o(1)$ and since $\log(1+t) = O((1+t)^{\ep/2})$, it follows that
$$\Phi(t,x) = o((\log n)^9(1+t)^{1+\ep})$$ uniformly in $t$, as $n\to\infty$.
\end{proof}

\noindent We tackle the proof of Proposition \ref{prop:branch-estimate} in a couple of steps.\\

\noindent\textbf{Step 1}. We obtain a somewhat crude upper bound on $(X_t)$ that has the virtue of being effective starting at time $0$. For $i\geq 1$ let $t_i = \inf\{t:X_t = i\}$, define $N_i = N_{t_i}^i$ then define $Y_t, \ Q_t$ by
\begin{equation}\label{eq1}
Y_0=1\quad\hbox{and}\quad Y_t \rightarrow Y_t+1\quad\hbox{at rate}\quad Q_t = b+\sum_{i=2}^{Y_t}1/(1+N_{t_i}^i).
\end{equation}
In words, at the moment $t_i$ an individual $i$ is added to the process, the corresponding counting process $N_t^i$ is stopped, so that $i$ always contributes $(1+N_{t_i}^i)^{-1}$ to $Q_t$. Since $(1+N_{t_i}^i)^{-1} \geq (1+N_t^i)^{-1}$ for $t \geq t_i$, $(X_t)$ is dominated by $(Y_t)$. The next result controls $(Y_t)$.

\begin{lemma}\label{lem:small-time}
There is $M_1 \in [1,\infty)$ so that for $a\ge 2, \ b\ge 1$ and $r \ge 1/2$,
$$\P(\sup_{t \ge 0}Y_t/(1+rt)^{1+1/r} \ge abM_1) \leq 2e^{-(a-2)/2}$$
\end{lemma}

\begin{proof}
Begin by observing that $(Q_t)$ has the concise description
$$Q_0=b \quad\hbox{and}\quad Q_t \rightarrow Q_t + \Delta_t \quad\hbox{at rate}\quad Q_t$$
where the increment $\Delta_t \stackrel{d}{=} (1+\poi(rt))^{-1}$ is independently sampled every time there is a jump. Our first task is to control the size of $Q_t$. We compute the drift:
$$\mu_t(Q) = \ell(t) Q_t \quad\hbox{with}\quad \ell(t) := \E[\Delta_t]$$
Let $g(t) = \exp(\int_0^t \E[(1+\poi(rs))^{-1}]ds)$. Using Lemma \ref{lem:inhlin-drift} with $b(t)=0$ and $c=1$, for $a\ge 2$ we find
\begin{equation}\label{eq2}
\P(\sup_t Q_t/g(t)) \geq ab) \leq e^{-(a-2)b/4}
\end{equation}
This translates to a bound on $(Y_t)_{t\geq 0}$ as follows. Since $\mu_t(Y) = Q_t$,
$$Y_t^m = Y_t - Y_0 - \int_0^tQ_sds$$
Since $(Y_t)_{t\geq 0}$ has transition rate $Q_t$ and jump size exactly 1, $\sigma^2(Y_t) = Q_t$. Taking $\lambda=1/2$ in \eqref{eq:sm-est2} (which satisfies $c\lambda \le 1/2$) while noting $Y_0=1$,
$$\P(Y_t \geq 1 + a + \frac{3}{2}\int_0^t Q_sds\quad\hbox{for some}\quad t\geq 0) \leq e^{-a/2}$$
Combining with \eqref{eq2} and taking a union bound,
\begin{equation}\label{eq3}
\P(Y_t \geq 1 + a\left( 1 + \frac{3b}{2}\int_0^t g(s)ds \right)\quad\hbox{for some}\quad t\geq 0) \leq e^{-a/2} + e^{-(a-2)b/4}.
\end{equation}
Intuitively, $g(t)$ grows roughly like $m(t)$. Let $\xi = \poi(\lambda)$.
 Since $x\mapsto(1+x)^{-1}$ is convex, the inequality $\E[(1+\xi)^{-1}] \geq (1+\E[\xi])^{-1}$ goes in the wrong direction for an upper bound on $g(t)$.
 Anticipating our needs, we let $x = \lambda^{\alpha}/2$ in \eqref{eq:poi-md} to find
\begin{equation}\label{eq4}
\P( \xi < \lambda - \lambda^{1/2 + \alpha}/2) \leq e^{-\lambda^{2\alpha}/8} \quad\hbox{if}\quad 0<\alpha \leq 1/2.
\end{equation}
Using the fact that $(1+\xi)^{-1} \leq 1$ and that probabilities are at most $1$, then using \eqref{eq:recip-bound},
\begin{eqnarray*}
\E[(1+\xi)^{-1}] &=& \E[(1+\xi)^{-1} \ ; \ \xi \geq \lambda-\lambda^{1/2 + \alpha}/2] + \E[(1+\xi)^{-1} \ ; \ \xi < \lambda-\lambda^{1/2 + \alpha}/2] \\
& \leq & (1+\lambda -\lambda^{1/2 + \alpha}/2)^{-1}\P(\xi \geq \lambda-\lambda^{1/2+ \alpha}/2) + \P(\xi < \lambda-\lambda^{1/2 + \alpha}/2) \\
& \leq & (1+\lambda -\lambda^{1/2 + \alpha}/2)^{-1} + e^{-\lambda^{2\alpha}/8} \\
& \le & (1+\lambda)^{-1} + (1+\lambda)^{-3/2+\alpha} + e^{-\lambda^{2\alpha}/8} 
\end{eqnarray*}
Also, if $0<\alpha<1/2$ and $0<r\le 1$ then
$$c(r,\alpha) := \int_0^{\infty}((1+rs)^{-3/2+\alpha} + e^{-(rs)^{2\alpha}/8})ds < \infty.$$
Let $c(r) = \inf \{c(r,\alpha):\alpha \in (0,1/2)\}$ and let $c=c(1/2)$. Since $c(r,\alpha)$ decreases with $r$, it follows that $c(r) \le c$ for $r \ge 1/2$. Recalling $m(t) = \exp(\int_0^t ds/(1+rs))$ defined earlier, it follows that
\begin{eqnarray*}
g(t) & \leq & \inf_{\alpha \in (0,1/2)}\exp \left( \int_0^t((1+rs)^{-1} + (1+rs)^{-3/2+\alpha} + e^{-(rs)^{2\alpha}/8})ds \right) \\
& \leq & e^c m(t)
\end{eqnarray*}
Since $m(t) = (1+rt)^{1/r}$ and $1/(r(1+1/r)) = 1/(r+1)$, and since $r>0$, it follows that
$$\begin{array}{rcl}
\int_0^t g(s)ds \le e^c \int_0^t(1+rt)^{1/r} &=& e^c((1+rt)^{1/r+1}-1)/(r+1) \\
& \le & e^c((1+rt)^{1/r+1}-1).
\end{array}$$
If $a\ge 2$ and $b\ge 1$ then since $c>0$, $1+a(1-3be^c/2) \le 0$ and
$$1 + a \left(1 + \frac{3b}{2}\int_0^t g(s)ds \right)\le \frac{3abe^c}{2}(1+rt)^{1/r+1}.$$
To conclude, take $M_1 = \frac{3}{2}e^c$, use \eqref{eq3} and note $e^{-a/2}+e^{-(a-2)b/4} \le 2e^{-(a-2)/2}$ for $b\ge 2$.
\end{proof}

\noindent\textbf{Step 2.} Next, we do two things.
\begin{enumerate}[noitemsep]
\item \textit{Lemma \ref{lem:br-env}}. We control the environment $\{(N_t^i)\}_{i\ge 1}$ for $t \in [T,\infty)$.
\item \textit{Lemma \ref{lem:br-late}}. We use this to get an upper bound on $(X_t)$ for $t \in [T,\infty)$.
\end{enumerate}
Fix $\alpha \in (0,1/2)$, then let
$$\tau_{lp}(i) = \sup\{t:N_t^i - rt + (rt)^{1/2+\alpha}/2 < 0\} \quad \hbox{for} \quad i\ge 1$$
denote the last passage time of $N_t^i$ below the curve $v(t) = rt-(rt)^{1/2+\alpha}/2$, and for $t\ge 0$ let
$$I_t = \max\{i: \tau_{lp}(j) \le t \ \hbox{for all} \ j \le i\}.$$
Note that $\Lambda_t(x) \le b + x/(1+v(t))$ for $x \le I_t$.

\begin{lemma}\label{lem:br-env}
There is $T_0>0$ so that for $r \ge 1/2$ and $\alpha \in(0,1/2)$,
$$\P(\inf_{t > T} I_t - t^{1/2-\alpha}e^{(rt)^{2\alpha}/4} < 0) \le 17T^{3/2-3\alpha}e^{-(rT)^{2\alpha}/4} \quad \hbox{for} \ \quad T>T_0.$$
\end{lemma}

\begin{proof}
For each $i$, using Lemma \ref{lem:poi-proc-md} with $\lambda=r$ and $\tau_{lp}(i)=\tau_2$,
$$\P(\tau_{lp}(i) \ge t) \le 4t^{1-2\alpha}e^{-(rt)^{2\alpha}/2} \quad \hbox{if} \quad r\ge 1,t^{2\alpha}\ge 4.$$
Let $f(t) = 2t^{1/2-\alpha}e^{-(rt)^{2\alpha}/4}$, so the right-hand side above is $1/f(t)^2$.
 Then, a union bound and the fact that $f(t)^{-1} \le 1$ gives
$$\begin{array}{rcl}
\P(I_t < f(t)) = \P(\max_{j \le \lceil f(t) \rceil }\tau_{lp}(j) > t) &\le & \lceil f(t) \rceil /f(t)^2 \\
&\le & f(t)^{-1}(1+ f(t)^{-1}) \le 2f(t)^{-1}.
\end{array}$$
For $T>0$ let
$$\begin{array}{rcl}
 c_1 &=& \sup_{t \ge T}f(t)/f(t+1) \quad \hbox{and} \quad \\
 c_2 &=& (1 - 4r^{-2\alpha}(3/2-3\alpha)T^{-2\alpha})^{-1}.
 \end{array}$$
 Note that $c_1,c_2 \in (0,\infty)$ and $\lim_{T\to\infty}c_1,c_2 = 1$ uniformly for $r \in [1/2,1]$.
 Since $I_t$ is non-decreasing, if $I_t \ge f(t)$ and $t>T$ then
 $$I_{t+h} \ge f(t) \ge c_1f(t+h) \quad \hbox{for} \quad h \in [0,1).$$
Taking a union bound over the estimate at times $T+k$, $k\ge 0$ gives
$$\P(\inf_{t > T}I_t - c_1f(t) < 0 ) \le \sum_{k\ge 0}4(T+k)^{1/2-\alpha}e^{-(r(T+k))^{2\alpha}/4}.$$
The right-hand side is at most
$$4r^{-2\alpha}c_1\int_T^{\infty}4t^{1/2-\alpha}e^{-(rt)^{2\alpha}/4}dt,$$
and using \eqref{eq:exp-asym}, this is at most
$$16r^{-2\alpha}c_1c_2 T^{3/2-3\alpha}e^{-(rT)^{2\alpha}/4}.$$
Then note that $c_1 \ge 1/2$ and $16r^{-2\alpha}c_1c_2 \le 17$ for $T$ large enough, uniformly for $r \in [1/2,1]$.
\end{proof}

\begin{lemma}\label{lem:br-late}
Given $\alpha\in (0,1/2)$ let $\tau = \inf\{t>T:X_t > I_t\}$. There is $M_2 \in [1,\infty)$ so that for $r \in [1/2,1)$ and $a\ge 2$,
$$\begin{array}{rl}
\P(\sup_{T \le t < \tau}X_t - M_2(ax/(1+rT)^{1/r} + 2b\log(1+t))(1+rt)^{1/r} > 0\ \mid & X_T \le x) \\
& \le e^{-(a-2)x/4(1+b(1+rT))}
\end{array}$$
\end{lemma}

\begin{proof}
Since, as noted before, $\Lambda_t(y) \le b + y/(1+v(t))$ for $y \le I_t$, it follows that for $T \le t< \tau$ and conditioned on $X_T \le x$, $(X_t)$ is dominated by the process $\tilde{X}_t$ with $\tilde{X}_T = x$ that increases by $1$ at rate $b + \tilde{X}_t/(1+v(t))$.
 We proceed as in the proof of Lemma \ref{lem:small-time}.
 We have
$$\mu_t(\tilde{X}) = b + \ell(t) \tilde{X}_t \quad\hbox{with}\quad \ell(t) := 1/(1+v(t)).$$
For $a>0$ let $E_a = \{\sup_{t \ge T}\tilde{X}_t/g(t) \ge ax + b\int_T^tds/g(s)\}$, where $g(t) = \exp( \int_T^t \ell(s)ds)$, and let $\beta = b\int_T^{\infty}ds/g(s)^2$. Using Lemma \ref{lem:inhlin-drift} with $c=1$, for $a\ge 2$ we find
\begin{equation*}
\P( E_a ) \leq e^{-(a-2)x/4(1+\beta)}.
\end{equation*}
Recall $v(t) = rt-(rt)^{1/2+\alpha}/2$. Using \eqref{eq:recip-bound} with $\lambda = rt$,
$$\ell(t) = 1/(1+v(t)) \le (1+rt)^{-1} + (1+rt)^{-3/2+\alpha}.$$
Let $c(r) = \int_0^{\infty} (1+rs)^{-3/2+\alpha}ds$, which is finite for $\alpha \in (0,1/2)$ and $r\in (0,1]$ and decreases with $r$.
 Let $c=c(1/2)$, so that $c(r) \le c$ for $r \in [1/2,1]$. Combining and noting $1+rt \le 1+t$,
$$g(t) \le e^c \exp\left(\int_T^t ds/(1+rs) \right) \le e^c (1+t)^{1/r}/(1+rT)^{1/r}.$$
Using that $\ell(t) \ge 1/(1+rt)$, we obtain the complementary bound $g(t) \ge ((1+rt)/(1+rT))^{1/r}$. In this way
$$\begin{array}{rcl}
   \beta = b \int_T^{\infty}ds/g(s)^2 &\le & (1+rT)^{2/r}\int_T^{\infty}(1+rs)^{-2/r}ds \\
   &=& b(1+rT)^{2/r}(2/r-1)^{-1}(1+rT)^{1-2/r} \le b(1+rT).
  \end{array}$$
Using the more generous lower bound $g(t) \ge (1+rt)/(1+rT)^{1/r}$ and noting $\log(1+rT)\ge 0$ and $1/2 \le r \le 1$, we find
$$(1+rT)^{-1/r}b\int_T^t ds/g(s) \le b\int_T^tds/(1+rs) \le 2b\log(1+rt) \le 2b\log(1+t).$$
Let $M_2 = e^c$ and rearrange terms in the formula for $E_a$ to complete the proof.
\end{proof}

\begin{proof}[Proof of Proposition \ref{prop:branch-estimate}]
We note the result of Lemma \ref{lem:small-time} applies to $(X_t)$ since it is dominated by $(Y_t)$.
 Fix $\alpha=1/4$ in Lemma \ref{lem:br-env} and \ref{lem:br-late}.
 Fix $T>0$ and let $L(t) = \frac{1}{2}e^{(rt)^{1/2}/8}$ and $x_1 = abM_1(1+rT)^{1+1/r}$. Let
$$\begin{array}{rcl}
E &=& \{\sup_{t \le T}X_t/(1+rt)^{1+1/r} \le abM_1\}, \\
F &=& \{\inf_{t \ge T} I_t - L(t) \ge 0\} \quad \hbox{and} \\
G &=& \{\sup_{T \le t < \tau} X_t - (aM_2x_1/(1+rT)^{1/r} + 2bM_2\log(1+t))(1+rt)^{1/r} \le 0\}
\end{array}$$
be the complement of the event from, respectively, Lemma \ref{lem:small-time}, \ref{lem:br-env} and \ref{lem:br-late}. On $E$,
$$\sup_{t \le T} X_t/(1+rt)^{1/r} \le abM_1(1+rT).$$
In particular, $X_T \le x_1$, so using Lemma \ref{lem:br-late}, for $b\ge 1$ and $T$ large enough,
\beq\label{eq:GcEprob}
\begin{array}{rcl} \P( G^c \cap E ) &\le & \P(G^c \cap \{X_T \le x_1\}) \le \P(G^c \mid X_t \le x_1) \\
&\le & e^{-(a-2)x_1/4(1+b(1+rT))} \le e^{-(a-2)aM_1(1+rT)^{1/r}/5}.
\end{array}\eeq
Using our choice of $x_1$, on $G$ we find
$$\begin{array}{rl} &\sup_{T \le t < \tau} X_t - (M(a,T)+2bM_2\log(1+t))(1+rt)^{1/r} \le 0 \quad \hbox{with} \\
& M(a,T) = a^2bM_1M_2(1+rT).
\end{array}$$
Since $a,M_2 \ge 1$, on $E\cap G$ the inequality holds for all $t < \tau$.
 Taking $a=T^{1/2}$, since $r\ge 1/2$,
 $$M(a,T) \le TbM_1M_2(1+rT) \le bM\cdot (1+rT)^2 \quad \hbox{with} \quad M = 2M_1M_2.$$
Using Lemmas \ref{lem:small-time} and \ref{lem:br-env}, for $T>T_0$ and some $T_0>0$, we find $x_1 \ge 1$ and
$$\begin{array}{rlcl}
& \P(E^c) &\le & 2ee^{-T^{1/2}/2} \quad \hbox{and} \\
& \P(F^c) &\le & 17T^{3/4}e^{-(rT)^{1/2}/4}.
\end{array}$$
If $T$ is large enough uniformly for $r \in [1/2,1]$ the above and \eqref{eq:GcEprob} show \\
$\P(E^c),\P(G^c \cap E) \le \frac{1}{2}T^{3/4}e^{-(rT)^{1/2}/4}$, so a union bound gives
$$\begin{array}{rcl}
\P((E \cap F \cap G)^c) \le \P(F^c) + \P(G^c) + \P(E^c) + \P(G^c \cap E) &\le & 18T^{3/4}e^{-(rT)^{1/2}/4} \\
&\le & 19(1+rT)^{3/4}e^{-(1+rT)^{1/2}/4}.
\end{array}$$
Choose $T$ so that for $x$ from the statement of the Proposition, $x=(1+rT)^2$. It suffices to check that $\tau=\infty$ on $E\cap F \cap G$. But on $E\cap F \cap G$, noting that $2M_2 \le M$ and $r\le 1$ on the second line,
$$\begin{array}{rcl}
I_t &\ge & L(t) = t^{1/4}e^{(rt)^{1/2}/4} \quad \hbox{and} \\
X_t &\le & bM(1+rt)^{1/r}((1+rT)^2 + \log(1+t)) \quad \hbox{for} \quad t \ge T.
\end{array}$$
Since $\tau = \inf\{t>T:X_t > I_t\}$ and by assumption, $b \le x = (1+rT)^2$, it suffices that\\
$$L(t) \ge M(1+rT)^2((1+t)^{1/r}((1+rT)^2 + \log(1+t))$$
for $t \ge T$, which is true for $T$ large enough, uniformly for $r \in [1/2,1]$.
\end{proof}

\subsubsection{Estimation of $T_{C,R}$}
Write $T_{C,R} = T_C \wedge T_R$, where
$$\bay{rcl}
T_C &=& \inf\{t:\max_w \sum_{v \in \Cu_t(w) \cap H_t}N_t(v)^{-1} \ge C\} \quad \hbox{and} \\
 \quad T_R &=& \inf\{t:\max_w \Ru_t(w) \ge R\}.
 \eay$$

\begin{proposition}\label{prop:TR}
Let $b=1+C$. If $R \ge b + (\log n)^{11}$ and $b \le (8 \log n)^4$ then
$$\lim_{n\to\infty}\P(T_R \le n^{1/2} \wedge T_C) = 0.$$
\end{proposition}

\begin{proposition}\label{prop:TC}
For each $\ep>0$ and $b=(8\log n)^4$, if $R \le (\log n)^{12}$ then
$$\lim_{n\to\infty}\P(T_C \le n^{1/2-\ep} \wedge T_R) = 0.$$
\end{proposition}

\begin{proof}[Proof of Proposition \ref{prop:TCR}]
Notice that
$$
T_{C,R} \le t \ \ \Leftrightarrow \ \ T_C \le t \wedge T_R \ \ \hbox{or} \ \ T_R \le t \wedge T_C,
$$
then use Propositions \ref{prop:TR} and \ref{prop:TC} and take a union bound.
\end{proof}

\noindent Next we prove Proposition \ref{prop:TR}, which is the simpler of the two.

\begin{proof}[Proof of Proposition \ref{prop:TR}]
Since for $t< T_{C,R}$, each $\Ru_t(w)$ is dominated by $\Lambda_t(X_t)$, taking a union bound we find
$$\P(T_R \le t\wedge T_C) \le n\P(\sup_{s \le t}\Lambda_s(X_s) \ge R).$$
For a given function $\Phi(t)$,
$$\{\sup_{s \le t}\Lambda_s(X_s) \ge R\} \subset \{\sup_{s \ge 0}X_s - \Phi(s) > 0\} \cup \{\sup_{s\le t}\Lambda_s(\Phi_s) \ge R\}.$$
Taking $x = (8 \log n)^4$ in Proposition \ref{prop:branch-estimate},
$$\bay{rl}
& \P( \sup_{s\ge 0} X_s - \Phi(s) >0) = o(1/n), \ \hbox{where} \\
& \Phi(s) = M(8\log n)^4((8 \log n)^4 + \log(1+s))(1+s)^{1/r}. \eay$$
We have the trivial bound $\Lambda_s(x) \le b + x$, and it follows from the bound on $\ell(t)$ given in the proof of Lemma \ref{lem:br-late} that
$$
\Lambda_s(x) \le b + 2x(1+rs)^{-1} \ \hbox{for} \ x \le I_t.
$$
Let $L(s) = s^{1/4}e^{(rs)^{1/2}/4}$. Taking $T = 2(8\log n)^2$, if $n$ is large enough uniformly for $r \in [1/2,1]$ then $\Phi(s) \le L(s)$ for $s \ge T$ so using Lemma \ref{lem:br-env} and the above bounds, for $1/2\le r \le 1$ and large $n$ we find
$$\P(\sup_{s \ge 0}\Lambda_s(\Phi(s)) - (b + \Phi((2(8\log n)^2) \vee (2\Phi(s)(1+rs)^{-1})) > 0) = o(1/n)$$
if $T$ is large enough. If $R = o(n/\log n)$, then noting $r=1-R/(n-1)$ and $r\ge 1/2$ for large $n$, if $s \le n^{1/2}$ then for large $n$,
$$(1+s)^{1/r}/(1+rs) \le 2(1+s)^{1/r-1} \le 2n^{2R/(n-1)} = 2e^{o(1)}$$
which approaches $2$ as $n\to\infty$. It follows that
$$\sup_{s \le n^{1/2}}2\Phi(s)(1+rs)^{-1} = O((\log n)^8),$$
and a similar estimate shows that $\Phi(2(8\log n)^2) = O((\log n)^{10})$. The result follows.
\end{proof}

It remains to prove Proposition \ref{prop:TC}. Define the non-decreasing spacetime set of points
$$\A_t(w) = \left \{(v,s):s \le t \ \hbox{and either}  
\begin{array}{rl}
& v \in \Cu_{s^-}(w) \ \hbox{and agreement occurs at} \ (v,\cdot,s) \ \hbox{or} \ (\cdot,v,s), \ \hbox{or} \\
& v \in H_s(w) \cap \Cu_s(w) \setminus \Cu_{s^-}(w).
\end{array}
\right \}$$
To get a more workable quantity we will use the fact that
$$\sum_{v \in \Cu_t(w) \cap H_t} N_t(v)^{-1} \le C_t(w) = \sum_{(v,s) \in \A_t(w)}1/(1 + N_{t-s}^{\ell}(v)).$$
This way,
\beq\label{eq:Cfact}
\hbox{if} \ \sup_{s \le t}\max_w C_s(w) < C \ \ \hbox{then} \ T_C>t.
\eeq
So, to estimate $T_C$ we control contributions to $C_t(w)$. Let $Q_t(w)$ denote the rate at which $\A_t(w)$ increases. Let $\{N_t^i:t \ge 0, \ i\ge 1\}$ be an independent collection of Poisson processes with intensity 1, let $Q,T>0$ and let $N(t)$ be an independent Poisson process with intensity $Q$. Let $t_i = \inf\{t:N(t)=i\}$ and let
$$B_t = \sum_{i \le N(t)}1/(1+N_{t-t_i}^i).$$
Let $T_Q(w) = \inf\{t:Q_t(w) > Q\}$ and $T_Q = \min_w T_Q(w)$. Then for any $w$,
\beq\label{eq:BdomC}
(C_t(w))_{t \le T_Q} \ \hbox{is stochastically dominated by} \ (B_t).
\eeq
In the next lemma we control $B_t$.

\begin{lemma}\label{lem:Bbnd}
For $T>0$, $T_0\ge 1$ and $Q\ge 1$,
$$\P(\sup_{t \le T} B_t > 2QT_0 + 4Q\log(2 \vee T)) \le Q(2+T)^2e^{-T_0/16}.$$
\end{lemma}

\begin{proof}
We first control the value of $B_T$, then of $B_t$ for $t \in [T-1,T]$, then take a union bound to control the value over the interval $[0,T]$.
 Let $\tilde N(t) = N(T) - N(T-t)$ and fix $T_0 < T$. 
 Using \eqref{eq:sm-est2} with $X_t = \tilde N(t)$, $X_t^p = \langle X^m \rangle_t = Qt$, $c =1$, $a=QT_0/2$, $\lambda = 1/2$ and $\bullet = +$ and noting that $c\lambda \le 1/2$,
$$\P( \tilde N(t) \ge 2Q(t\vee T_0) \ \hbox{for some} \ t \le T) \le e^{-QT_0/4} \le e^{-T_0/4}$$ 
 since $Q\ge 1$. Using the same result except with $X_t = N_t^i$, $X_t^p = \langle X^m \rangle_t = t$, $\bullet = -$, $a=T_0/4$ and $\lambda = 1/4$, and taking a union bound,
 $$\P(N_t^i < t/2 \ \hbox{for some} \ t\ge T_0 \ \hbox{and} \ 2QT_0 < i \le 2QT) \le 2Q(T-T_0)e^{-T_0/16}.$$
 Let $S = \{t \le T: \tilde N(t) > \tilde N(t^-)\}$ be the jump times of $\tilde N(t)$, and label them in increasing order as $\tilde t_1, \tilde t_2,\dots,\tilde t_{\tilde N_T}$.
 On the complement of both events above, $\tilde N(T) < 2QT$ and $\tilde t_i > i/2Q$ for $i \ge 2QT_0$, and so $N_{\tilde t_i}^i \ge i/4Q$, and this gives
 $$\begin{array}{rcl}
 B_T & \le & 2QT_0 + \sum_{2QT_0 \le i < 2QT}1/(1+i/4Q)\\ \\
 & \le & 2QT_0 + \int_{2QT_0}^{2QT}(1+t/4Q))^{-1}dt\\ \\
 &=& 2QT_0 + 4Q(\log(1+T/2) - \log(1+T_0/2)) \le 2QT_0 + 4Q\log(2 \vee T).
 \end{array}$$
To see that this also bounds $B_t$ for $t \in [T-1,T]$, replace $N_{\tilde t_i}$ with $N_{\tilde t_i-1} \ge (i/2Q-1)/2$ and use in the above to obtain the bound
$$2QT_0 + \int_{2QT_0}^{2QT}(1+(t/2Q-1)/2)^{-1}dt \le 2Q(T_0 + 2(\log(1/2+T/2) - \log(1/2+T_0/2))$$
which has the same upper bound, assuming $T_0 \ge 1$ so that $\log(1/2+T_0/2) \ge 0$. The same works for $T_1<T$ and $T_0\ge 1$ to give
$$\P(B_t > 2QT_0 + 4Q\log T_1 \ \hbox{for some} \ t \in [T_1-1,T_1]) \le 2Q(T_1-T_0)e^{-T_0/16} + e^{-T_0/4} \le 2QT_1e^{-T_0/16}$$
since $2QT_0e^{-T_0/16} \ge e^{-T_0/4}$. Taking a union bound over $T_1=1,2,\dots,\lfloor T \rfloor,\lfloor T \rfloor + 1$ and noting $\lfloor T \rfloor \le T$ gives the result.
\end{proof}

\noindent It remains to prove the following result.

\begin{proposition}\label{prop:TQ}
For small $\ep>0$ and any $k>0$, $Q=\log n$ and $R \le (\log n)^k$,
$$\lim_{n\to\infty}\P(T_Q \le n^{1/2-\ep} \wedge T_{C,R})=0.$$
\end{proposition}

\noindent Before proving it, we show how it implies Proposition \ref{prop:TC}. Use whp (with high probability) to denote an event whose probability tends to 1 as $n\to\infty$. Note that if $E_1,E_2$ whp then $E_1 \cap E_2$ whp.

\begin{proof}[Proof of Proposition \ref{prop:TC}]
We want to show that $T_C > n^{1/2-\ep} \wedge T_R$ whp. Since $T_C \ge T_C \wedge T_Q$, if $T_C \wedge T_Q > n^{1/2-\ep} \wedge T_R$ then $T_C > n^{1/2-\ep} \wedge T_R$. Moreover
$$T_C \wedge T_Q > n^{1/2-\ep}\wedge T_R \Leftrightarrow T_C > n^{1/2-\ep} \wedge T_R \wedge T_Q \ \ \hbox{and} \ \ T_Q > n^{1/2-\ep} \wedge T_R \wedge T_C.$$
Proposition \ref{prop:TQ} says that $T_Q > n^{1/2-\ep} \wedge T_R \wedge T_C$ whp, so it is enough to show that if $b=(8\log n)^4$ and $Q = \log n$ then $T_C > n^{1/2-\ep}\wedge T_R \wedge T_Q$ whp, or equivalently that

$$\P(T_C \le n^{1/2-\ep}\wedge T_R \wedge T_Q ) = o(1).$$

In Lemma \ref{lem:Bbnd} take $T=n^{1/2}$, $T_0 = 48 \log n$ and $Q = \log n$ to find that
$$\P(\sup_{t \le n^{1/2}} B_t > 98(\log n)^2) = O(n^{-2}\log n) = o(1/n).$$
Then, using \eqref{eq:BdomC} and Proposition \ref{prop:TQ} and taking a union bound over the $n$ possible values of $w$,
$$\P(\sup_{t \le n^{1/2-\ep} \wedge T_R \wedge T_Q} \max_w C_t(w) > 98(\log n)^2) = o(1).$$
The result then follows from \eqref{eq:Cfact} and the fact that $98(\log n)^2 < (8\log n)^4$ for large $n$.
\end{proof}

\noindent By taking a union bound over $w$ and noting the probability does not depend on $w$, to obtain Proposition \ref{prop:TQ} it is sufficient to show that for any $w$ and small $\ep>0$,
\beq\label{eq:Qbnd}
\lim_{n\to\infty}\P(\sup_{t \le n^{1/2-\ep}}Q_t(w) > \log n) = o(1/n),
\eeq
noting that the probability is the same for any $w$. There are three ways that $\A_t(w)$ increases:
\begin{enumerate}[noitemsep]
\item a site already in $H_t$ is added to $\Cu_t$,
\item agreement occurs at a site already in $\Cu_t$, or
\item a site is added simultaneously to $\Cu_t$ and $H_t$.
\end{enumerate}
Let $Q_t^i(w), \ i=1,2,3$ denote the rate of each event, so that $Q_t(w) = \sum_{i=1}^3Q_t^i(w)$. Since each site in $V\setminus \Cu_t(w)$ is added to $\Cu_t(w)$ at rate $\Ru_t(w)/(n-1)\le R/(n-1)$,
\beq\label{eq:Q1}
Q_t^1(w) \le |H_t| R/(n-1).
\eeq
Since there are $|\Cu_t(w)\cap \C_t(v)|$ sites in $\Cu_t(w)$ that can agree on word $v$, and each word is spoken at rate at most $R/(n-1)$ to each site,
\beq\label{eq:Q2}
\bay{rcl}
Q_t^2(w) &\le & \sum_v |\Cu_t(w) \cap \C_t(v)| R/(n-1) \\
&\le & \sum_{v \ne w}|\Cu_t(w) \cap \C_t(v)|R/(n-1) + S_tR/(n-1),
\eay\eeq
recalling that $S_t = \max_w |\C_t(w)|$ is the size of the largest cluster. Each time a person speaks, the probability that agreement occurs is at most $S_t/(n-1)$. Since $\Cu_t(w)$ increases at rate $\le R$, it follows that
\beq\label{eq:Q3}
Q_t^3(w) \le S_t R/(n-1).
\eeq
The reader may think that $Q_t^3(w)$ should be $0$, since a new addition to a cluster does not yet know the word. However, the upper bound cluster $\Cu_t(w)$ can grow when in the process itself, a word other than $w$ is being spoken. Using Proposition \ref{prop:branch-estimate} we control $Q_t^1(w)$ and $Q_t^3(w)$, which is two thirds of Proposition \ref{prop:TQ}.

\begin{lemma}\label{lem:Q13bnd}
For each $w$, small $\ep>0$, $R \le n^{\ep}$ and $i=1,3$,
$$\P(\sup_{t \le n^{1/2-\ep }\wedge T_{C,R}}Q_t^i(w) > 1)=o(1/n).$$
\end{lemma}

\begin{proof}
From \eqref{eq:Q1} and \eqref{eq:Q3} and the choice of $R$, it suffices to show that
$$\P(\sup_{t \le n^{1/2-\ep} \wedge T_{C,R}}\max S_t,|H_t| > n^{1-\ep}) = o(1/n).$$
The result of Proposition \ref{prop:max-cluster} holds with the probability being $o(1/n)$ --  to see this, take $x=(12 \log n)^4$ in the proof. This gives
$$\P(\sup_{t \le T_{C,R}}S_t - \Phi(t,x)>0) = o(1/n),$$
while $\Phi(t,x)$ is still $o((\log n)^9(1+t)^{1+\ep})$, uniformly in $t$ as $n\to\infty$. The desired result for $i=3$ then follows from \eqref{eq:Q3}, since $\sup_{t \le n^{1/2-\ep}}(\log n)^9(1+t)^{1+\ep} = (\log n)^9(1+n^{1/2-\ep})^{1+\ep} = o(n^{1-\ep})$. To get the result for $i=1$ recall from the beginning of this section that $|H_t| \le 2A_t$, the number of agreements up to time $t$, and from \eqref{eq:agree-rate} that $A_t \le \poi(\int_0^u S_udu)$. Using the above bound on $S_t$, with probability $1-o(1/n)$,
$$\bay{rcl}
\int_0^{n^{1/2-\ep} \wedge T_{C,R}}S_udu & \le & \int_0^{n^{1/2-\ep}} \Phi(u,x)du \\
&=& o((\log n)^9(1+n^{1/2-\ep})^{2+\ep}) = o((\log n)^9n^{1 - 3\ep/2 - \ep^2}) = o(n^{1-\ep})
\eay$$
for large $n$. From Lemma \ref{lem:mod-dev}, $\P(\poi(\lambda) > 2\lambda) \le e^{-\lambda/3}$, so it follows that
$$\P(\sup_{t \le n^{1/2-\ep} \wedge T_{C,R}}|H_t| > n^{1-\ep}) \le e^{-n^{1-\ep}/6} + o(1/n) = o(1/n).$$
\end{proof}

It remains to control $|\Cu_t(w)\cap \C_t(v)|$. We'll make use of the estimates from Lemma \ref{lem:Bbnd}, namely that for a Poisson process $N(t)$ with intensity 1,
\begin{equation}\label{eq:poigenbnd}\begin{array}{rcl}
\P(N_t \ge 2(t \vee T) \ \hbox{for some} \ t \ge 0) & \le & e^{-T/4} \ \hbox{and} \\
\P(N_t \le t/2 \ \hbox{for some} \ t \ge T) & \le & e^{-T/16}.
\end{array}
\end{equation}
First we modify slightly the construction from the beginning of Section \ref{subsec:maintain}, using a randomization trick.
 The reason it needs modifying is to ensure the growth of $\Cu_t(w)$ and any $\C_t(v)$ are not strongly correlated.
 Since we only randomize the location of ``excess'' events that expand $\Cu_t(w)$, the reader may verify that up to a random permutation of certain vertices, the marginal distribution of each $\Cu_t(w)$, and its domination of $\C_t(w)$, are unchanged.\\
 
 To carry out the modification, make the $\{U_v\}$ doubly-augmented, that is, each $U_v$ is again a Poisson point process with intensity 1, but on $[0,\infty) \times [0,1]^2$ instead of $[0,\infty)\times [0,1]$. $\Ru_t(w)$ is defined in the same way as before, and $\Cu_t(w)$ is defined as follows.
$$\begin{array}{rll}
&\hbox{if}& \ (t,u_1,u_2) \in U_v \ \hbox{and} \ u_1 \in I_{t^-}(v,i,R_{t^-}(w)/(n-1)), \\
&\hbox{or if}& \ (t,u_1,u_2) \in U_v, u_1 \notin I_{t^-}(v,i,R_{t^-}(w)/(n-1)) \\
&&\hbox{and} \ u_2 \le (n-1)^{-1}(\Ru_{t^-}(w)-R_{t^-}(w))/(1-R_{t^-}(w)), \\
&\hbox{then}& \Cu_t = \Cu_{t^-}\cup \{v\}.
\end{array}$$ 
In other words,
\begin{itemize}[noitemsep]
\item if $\C_t(w)$ was about to include $v$, then $\Cu_t(w)$ will too, and
\item if $\Cu_t(w)$ increases when $\C_t(w)$ does not, then with respect to \\ what
other clusters are doing, it does so as randomly as possible.
\end{itemize}
We now control the size of $\Cu_t(w) \cap \C_t(v)$, for any $v \ne w$. ``wp'' is shorthand for ``with probability''.

\begin{lemma}\label{lem:2clus-bd}
For any $\ep,k>0$, if $R \le n^{\ep/4}$ then
$$\P(\sup_{t \le n^{1/2-\ep} \wedge T_{C,R}}\sum_{v \ne w}|\Cu_t(w) \cap \C_t(v)| \ge n/(\log n)^k) = o(1/n).$$
\end{lemma}

\begin{proof}
Let $K_t=\sum_{v \ne w}|\Cu_t(w)\cap \C_t(v)|$. There are three ways $K_t$ can increase.
\begin{enumerate}[noitemsep]
\item $\Cu_t(w)$ acquires a site that belongs to some (possibly many) $\C_t(v)$, $v \ne w$,
\item some $\C_t(v)$, $v \ne w$ acquires a site that belongs to $\Cu_t(w)$, and
\item $\Cu_t(w)$ and some $\C_t(v)$, $v\ne w$ simultaneously acquire the same site.
\end{enumerate}
It suffices to show the contribution to $\sup_{t \le n^{1/2-\ep} \wedge T_{C,R}}K_t$ from each item is $o(n/(\log n)^k)$ wp $1-o(1/n)$. For item 1, the increase is at most $\max_v N_t^{\ell}(v)$, while for items 2,3 the increase is at most $1$, since at each transition, $C_t(v)$ increases for at most one $v$, and by at most $1$. Let $R_1(t),R_2(t),R_3(t)$ denote the rate of each event. 
 Then,
$$\begin{array}{rcl}
R_i(t) &\le & \Ru_t(w) \ \hbox{for} \ i \in \{1,3\},\ \hbox{and}\\
R_2(t) &\le & |\Cu_t(w)|\sum_{v \ne w}R_t(v)/(n-1).
\end{array}$$
Since $N_t^{\ell}(v) \le N_{n^{1/2}}^{\ell}(v)$ for $t \le n^{1/2}$ and each $v$, and since each $N_{n^{1/2}}^{\ell}(v) \sim \poi(n^{1/2})$, using \eqref{eq:poigenbnd} with $t=T=n^{1/2}$ and a union bound,
$$\P(\sup_{t \le n^{1/2}}\max_v N_t^{\ell}(v) > n^{1/2+\ep/2}) \le ne^{-n^{\ep/2}/4} = o(1/n).$$
For $t<T_{C,R}$, $\Ru_t(w) \le R$, so wp $1-o(1/n)$, the contribution from item 1 is at most $\poi(n^{1/2+\ep/2}Rn^{1/2-\ep}) = \poi(Rn^{1-\ep/4-\ep^2/2})$ which if $R \le n^{\ep/4}$ is at most $2n^{1-\ep^2/2} = o(n/(\log n)^k)$ for any fixed $k$ wp $1-o(1/n)$. Similarly, but more simply since the increase per transition is 1, the contribution from item 3 is at most $\poi(Rn^{1/2-\ep})$ which is at most $n^{1/2}=o(n/(\log n)^k)$ wp $1-o(1/n)$.

For item 2, note that $\sum_{v \ne w}R_t(v) \le \sum_v R_t(v) = n$ and that for $t<T_{C,R}$, $\Cu_t(w)$ is dominated by $X_t$. Applying Proposition \ref{prop:branch-estimate}, bounding $\log(1+t)$ by $\log(1+n^{1/2})$ for $t\le n^{1/2}$ and using the trivial but convenient $n/(n-1) \le 2$ for $n\ge 2$, we find that for $x \ge 1+C$ large enough,
$$\P(\sup_{t< n^{1/2} \wedge T_{C,R}}R_2(t) \ge 2Mx(x+ \log(1+n^{1/2}))(1+t)^{1/r}) \le 19x^{3/4}e^{-x^{1/4}/4}.$$
Taking $x= (12\log n)^4$ the probability above is $o(1/n^2)$. Thus the contribution from item 2 is at most $\poi(f(n^{1/2-\ep}))$, where
$$\begin{array}{rcl}
f(t) &=& \int_0^t2M(12 \log n)^4((12\log n)^4+ \log(1+n^{1/2}))(1+s)^{1/r} ds \\
&\le & 4M(12 \log n)^8(1+t)^{1+1/r},
\end{array}$$
using $\log(1+n^{1/2}) \le (12\log n)^4$ and $1+1/r \ge 1$. If $R = o(n)$ then for any $\ep>0$, $1/r \le 1+\ep$ for large $n$. Therefore
$$f(n^{1/2-\ep}) = O((\log n)^{8}n^{(1/2-\ep)(2+\ep)}) = O((\log n)^8n^{1-3\ep/2-\ep^2})=o(n/(\log n)^k).$$
It follows as before that $\poi(f(n^{1/2-\ep})) = o(n/(\log n)^k)$ wp $1-o(1/n)$, and the proof is complete.
\end{proof}

Combining this with the other term in \eqref{eq:Q2} we control $Q_t^2(w)$.

\begin{lemma}\label{lem:Q2bnd}
For any $k>0$ and small $\ep>0$, each $w$ and $R \le (\log n)^k$,
$$\P(\sup_{t \le n^{1/2-\ep} \wedge T_{C,R}}Q_t^2(w) > 2)=o(1/n).$$
\end{lemma}

\begin{proof}
From the proof of Lemma \ref{lem:Q13bnd} we know that $\P(\sup_{t \le n^{1/2-\ep} \wedge T_{C,R}}S_t > n^{1-\ep}) = o(1/n)$. Using this, \eqref{eq:Q2}, $R \le (\log n)^k$ and Lemma \ref{lem:2clus-bd}, 
$$\P(\sup_{t \le n^{1/2-\ep} \wedge T_{C,R}}Q_t^2(w) > (n/(\log n)^k) + n^{1-\ep})(\log n)^k/(n-1) ) = o(1/n).$$
If $n$ is large then $(n/(\log n)^k + n^{1-\ep})(\log n)^k/(n-1) \le 2$ and the result follows.
\end{proof}

\begin{proof}[Proof of Proposition \ref{prop:TQ}]
This follows from \eqref{eq:Qbnd}, and Lemmas \ref{lem:Q13bnd} and \ref{lem:Q2bnd}.
\end{proof}

\section{Final phase}\label{sec:final}

\subsection{Markov chain and ODE heuristic}

\noindent Using the notation of chemical reactions, we describe the eight types of interactions between any pair of individuals in Table~\ref{tab2}.
 Using this as a reference, we write down the eight transitions for the three coordinates of our Markov chain as well as for~$u = |x - y|$ which, as we will see later, is a
 key quantity in our analysis in Table~\ref{tab1}.
 Note that we have rescaled~$(X_t, Y_t, Z_t)$ to~$(x_t, y_t, z_t) = n^{-1} (X_t, Y_t, Z_t)$.
 Also note that~$\Delta_i (\dots)$ and~$q_i$ are respectively the change in quantity~$\dots$ and the transition rate at the~$i^{th}$ transition. \\
\begin{table}
\centering
\begin{tabular}{l l l l}
reactants & & product & $n\cdot$(rate)\\
\hline \vspace*{-5pt} \\
$A + AB$ &$\to$ & $2AB$ & $1/2$\\
$A + AB$ &$\to$ & $2A$ & $3/2$\\
$B + AB$ &$\to$ & $2AB$ & $1/2$\\
$B + AB$ &$\to$ & $2B$ & $3/2$\\
$AB + AB$ &$\to$& $2A$ & $1$\\
$AB + AB$ &$\to$& $2B$ & $1$\\
$A + B$ & $\to$& $B + AB$ & $1$\\
$A + B$ & $\to$& $A + AB$ & $1$
\end{tabular}\\
\caption{The various types of transitions occurring between two individuals}
\label{tab2}
\end{table}
\begin{table}
\centering
\setlength\tabcolsep{0.2in}
\begin{tabular}{r r r r l}
$n\Delta_i x$ & $n\Delta_i y$ & $n\Delta_i z$ & $n\Delta_i u$ & $n^{-1}q_i$ \\
\hline \vspace*{-5pt} \\
$-1$ & $0$ & $1$ & $-\sgn(x-y) + \mathbf{1}(u=0)$ & $xz/2$ \\
$1$ & $0$ & $-1$ & $\sgn(x-y) + \mathbf{1}(u=0)$ & $3xz/2$ \\
$0$ & $-1$ & $1$ & $\sgn(x-y) + \mathbf{1}(u=0)$ & $ yz/2$ \\
$0$ & $1$ & $-1$ & $-\sgn(x-y) + \mathbf{1}(u=0)$ & $ 3yz/2$ \\
$2$ & $0$ & $-2$ & $2(\sgn(x-y) + \mathbf{1}(n(x-y) \in \{0,-1\})$ & $ z(z-n^{-1})/2$ \\
$0$ & $2$ & $-2$ & $2(-\sgn(x-y) + \mathbf{1}(n(x-y) \in \{0,1\})$ & $ z(z-n^{-1})/2$ \\
$-1$ & $0$ & $1$ & $-\sgn(x-y) + \mathbf{1}(u=0)$ & $xy$ \\
$0$ & $-1$ & $1$ & $\sgn(x-y) + \mathbf{1}(u=0)$ & $xy$
\end{tabular}
\caption{List of transitions with jumps $\Delta_i$ and rates $q_i$}
\label{tab1}
\end{table}
 Note that~$x_t + y_t + z_t = 1$ for~$t \geq 0$.
 To get an idea of what to expect, notice that as~$n \to \infty$, sample paths approach solutions to the ODE system with~$z = 1 - (x + y)$ and
\begin{eqnarray}
\label{eq1}
 x' & = & xz + z^2 -xy \\ \nonumber
 y' & = & yz + z^2 -xy \nonumber
\end{eqnarray}
 that has the invariant set
 $$ \Lambda: = \{(x, y) \in \R_+^2 : x + y \leq 1 \}. $$
 The subset~$\ell = \{(x, y) \in \Lambda : x = y \}$ is also invariant, since if~$x = y$ then~$(x - y)'= (x - y) z = 0$.
 Adding the~$x'$ and~$y'$ equations, the dynamics on~$\ell$ is described by
 $$ z' = (1/2)(1 -4z - z^2) $$
 that has the stable fixed point~$z^* = - 2 + \sqrt{5}$.
 Thus, \eqref{eq1} has the equilibrium point
 $$ \bigg(\frac{1 - z^*}{2}, \frac{1 - z^*}{2} \bigg) = \bigg(\frac{3 - \sqrt{5}}{2}, \frac{3 - \sqrt{5}}{2} \bigg) $$
 whose stable manifold contains~$\ell$.
 For the dynamics off~$\ell$, let~$u = |x - y|$ as defined above, taking values in~$[0, 1]$.
 From \eqref{eq1}, we derive
\begin{eqnarray}
\label{eq2-1}
  u' & = & uz \\ \nonumber
  z' & = & (1/2)(1 - u^2 - 4z - z^2) \nonumber
\end{eqnarray}
 We see that~$u$ is non-decreasing, so~$u(\infty) := \lim_{t \to \infty} u(t)$ is well-defined.
 If~$u (0) > 0$ then~$u (\infty) > 0$, and if in addition~$u (\infty) < 1$ then according to~\eqref{eq2-1}, $z (t)$ has a positive limit, which contradicts~$\lim_{t \to \infty} u'(t) = 0$,
 therefore we must have~$u (\infty) = 1$. \\
\indent To see the connection to Theorem~\ref{thm:finalphase}, note that if~$u (0) \geq n^{-1}$, then since the eigenvalues of the linearization near~$u = 0$ and~$u = 1$ are both non-zero,
 it should take about constant times~$\log n$ amount of time for~$u$ to exceed~$1 - n^{-1}$.
 Next, we delve into the land of martingales to make this intuition precise. \\


\subsection{Controlling sample paths}
\label{sec4}

\noindent Defining the process $u$ by~$u_t = |x_t - y_t|$,
 we are interested in the time to consensus, that we can express as
 $$ \inf \,\{t : u_t = 1 \}.$$
Using the notation from just above, the drift and diffusivity take the form
 $$\mu(u)= \sum_i q_i(u)\Delta_i(u) \quad \hbox{and} \quad \sigma^2(u) = \sum_i q_i(u)\Delta_i^2(u).$$
 We'll write for now with $u$ but the same holds for $x,y,z$ and other functions of the state variables.
 For efficiency of notation, we'll allow the function to change depending on the variable, so $\mu(u)$ is different from $\mu(x)$ and $\mu(y)$. Also, instead of the compensator $u^p$ we'll use the \emph{predictor} $\bar u = u_0+u^p$ which includes the initial value, and we'll denote $u^m$ by $M(u)$, and $\langle u^m \rangle$ simply by $\langle u \rangle$. Then, $M(u) = u-\bar u$, and $\bar u$ and $\langle u \rangle$ can be written
 $$\bar u_t = u_0 + \int_0^t\mu_s(u)ds \quad \hbox{and} \quad \langle u \rangle_t = \int_0^t \sigma^2_s(u)ds.$$
Define the \emph{jump size} $c_\Delta(u)=\sup_{u,i} |\Delta_i(u)|$. From \eqref{eq:sm-est2}, if $a>0$, $0<c_\Delta(u)\lambda \le 1/2$ and $\bullet \in \pm$ then
$$\P(\bullet(u_t - \bar u_t) \ge a + \lambda\langle u \rangle_t \ \hbox{for some} \ t \ge 0) \le e^{-\lambda a}.$$
 Defining the maximum transition rate $c_q(u) = \sup_u \sum_i q_i(u)$, we have the basic inequality $\sigma^2(u) \le c(u):=(c_q c_\Delta^2)(u)$ and we obtain the corollary
 $$\P(\bullet(u_t - \bar u_t) \ge a + \lambda c(u) t \ \hbox{for some} \ t \ge 0) \le e^{-\lambda a}.$$
 For any quantity $\cdot$, we always have $c_q(\cdot) \le n$, since there are $n(n-1)$ directed edges each ringing at rate $1/(n-1)$, and for most quantities of interest, $c_\Delta(\cdot) \le jn^{-1}$ for a smallish integer $j$, giving $c_qc_\Delta^2 \le j^2n^{-1}$, allowing us to take $\lambda$ equal to a small multiple of $n$ while still keeping $\lambda c_qc_\Delta^2 t = O(1)$. When the context is clear, we omit the variable and simply write $c,c_q,c_\Delta$.\\
 
 The workflow of estimates is as follows. For any $\alpha>0$, we find $\ep>0$ so that the following holds with probability $1-o(1)$ as $n\to\infty$. Item numbers correspond to the Lemmas where they are proved.
\begin{enumerate}
\item So long as $u_t \leq 2\ep$, get $|z_t-z^*|< 2\ep$ within constant time and keep $|z_t-z^*| < 3\ep$ for $n$ time.
\item So long as $|z_t-z^*| \leq 3\ep$, get $u_t > 2\ep$ within $((2z^*)^{-1} + \alpha)\log n$ time, and find initial conditions so that $u_t \leq 2\ep$ for at least $((2z^*)^{-1} - \alpha) \log n$ time. 
\item Once $u_t > 2\ep$, keep $u_t \geq \ep$ for $n^{1/2}$ amount of time.
\item So long as $u_t < 1-\ep$, get $z_t > \ep/4$ within constant time and keep $z_t \geq \ep/12$ for $n$ time.
\item So long as $u_t \geq \ep$ and $z_t \geq \ep/12$, get $u_t > 1-\ep$ within constant times $\ep^{-2}$ time.
\item Once $u_t > 1-\ep$, keep $u_t \geq 1-2\ep$ for $n^{1/2}$ time, and show that if $u_t \leq 1-\ep+2n^{-1}$ then so long as $u_t \geq 1-2\ep$, $u_t < 1$ for at least $(1-\alpha)\log n$ time.
\item So long as $u_t \geq 1-2\ep$, get $u_t = 1$ within $(1+ \alpha)\log n$ time.
\end{enumerate}
 Propositions~\ref{prop1}--\ref{prop3} stitch together Lemmas~\ref{lemma1}--\ref{lemma2}, Lemmas~\ref{lemma3}--\ref{lemma5}, and Lemmas~\ref{lemma6}--\ref{lemma7}, respectively.
 The combination of these propositions into the proof of Theorem~\ref{thm:finalphase} is given at the end of this section.
\begin{lemma}
\label{lemma1}
 Let~$b_t = z_t-z^*$ and let
 $$\tau = \inf\{t : u_t \geq 2 \ep \},\quad \tau_0 = \tau\wedge\inf\{t:|b_t|\leq\ep\}\quad\hbox{and}\quad\tau_1 = \tau\wedge\inf\{t:|b_t|\geq 3\ep\}.$$
 If~$\ep \leq 1/4$ and~$n \geq 8 / \ep^2$ then
 $$ \P ( \, \tau_0 \geq 2/\epsilon^2 \, ) \leq e^{-\ep^2n/64}$$
 and for integer~$N > 0$,
 $$ \P ( \, \tau_1 \leq \ep^2N/6 \quad\hbox{and}\quad |b_{\tau_1}| \geq 3\ep \ \mid \ |b_0| \leq \ep \, ) \leq 2 N e^{-\ep^4n/8}. $$
\end{lemma}
\begin{proof}
Using Table~\ref{tab1}, we find that
 $$ \mu (b) = (1/2)(-b (z + 2 + \sqrt{5}) - u^2) + 2 zn^{-1}. $$
 Since $c_\Delta(b) \le 2n^{-1}$ and $c_q(b) \le n$, $\sigma^2(b) \le (c_qc_\Delta^2)(b) \le 4n^{-1}$, and using the product rule from Lemma \ref{lem:qac-prod} on $b\cdot b$,
\begin{eqnarray*}
 \mu (b^2) & = & 2b\,\mu (b) + \sigma^2(b) \le 2b\,\mu (b) + 4n^{-1}\\
         & \leq & b (-b (z + 2 + \sqrt{5}) - u^2) + 4 (1 + zb) n^{-1}
\end{eqnarray*}
 If~$u < 2 \ep$ and~$|b| > \ep$ then, since~$\sqrt{5} \geq 2$ and~$z \geq 0$,
 $$ \begin{array}{rcl}
     |-b (z + 2 + \sqrt{5}) - u^2| & \geq & |b (z + 2 + \sqrt{5})| - u^2 \vspace*{4pt} \\
                                   & \geq & 4 |b| - u^2 \geq 4 \ep - 4 \ep^2 = 4 \ep (1 - \ep) \end{array} $$
 so if in addition~$\ep, 2n^{-1} / \ep^2 \leq 1/4$, then since $zb \leq 1$,
\begin{equation}
\label{eq13}
  \mu(b^2) \leq -4 \ep^2 (1 - \ep) + 8n^{-1} \leq -2 \ep^2.
\end{equation}
and since $b_t^2-b_0^2 \geq -1$, we find
$$b_t^2 - \bar b_t^2  = b_t^2 - b_0^2 - \int_0^t\mu_s(b^2)ds \geq -1 + 2\epsilon^2 t.$$
Moreover, if $b \in [0,1]$ then for $\delta \in \R$, $|(b+\delta)^2-b^2| = |2\delta b + \delta^2| \leq 2|\delta| + \delta^2$, which implies that $c_\Delta(b^2) \le 2(2n^{-1}) + (2n^{-1})^2 \leq 8n^{-1}$. So, we can take $\lambda \in (0,n/16)$ and $c = c_qc_\Delta^2 = 64n^{-1}$. Choosing $\lambda = \ep^2n/64$ gives $c\lambda t = \epsilon^2 t$. If $\tau_0>T$ we find
$$b_T^2 - \bar b_T^2  - c\lambda T \geq -1 + \epsilon^2 T.$$
Taking $T = 2/\ep^2$ gives a lower bound of $1$. So, taking $a=1$ and $\bullet = +$ gives the first statement.
 Next, let $\tau_2 = \inf\{t:|b_t - 2\ep| \geq \ep\,\,\hbox{or}\,\,u_t \geq 2\ep\}$. Using \eqref{eq13}, $\mu(b_t^2)\leq -2\ep^2$ for $t<\tau_2$. Thus, if $|b_0-2\ep| \leq n^{-1}$, $b_{\tau_2} \geq 3\ep$ then
\begin{eqnarray*}
b_{\tau_2}^2 - \bar b_{\tau_2}^2 & \geq & (3\ep)^2 - (2\ep +n^{-1})^2 + 2\epsilon^2 \tau_2 \\
& \geq & 5\epsilon^2 - n^{-1}(4\epsilon + n^{-1}) + 2\epsilon^2 \tau_2
\end{eqnarray*}
Taking $\lambda = \epsilon^2n/32$, $c\lambda t = 2\epsilon^2 t$. If $n\geq 5/\ep$ then $b_{\tau_2}^2 - \bar b_{\tau_2}^2 -c\lambda \tau_2\geq 4\epsilon^2$. Noting that $\P(\tau_2<\infty)=1$, then taking $a=4\epsilon^2$ and $\bullet = +$,
\begin{equation}\label{eq17}
\P(|b_{\tau_2}| \geq 3\ep \mid |b_0 - 2\ep| \leq n^{-1}) \leq e^{-\ep^4 n/8}
\end{equation}
On the other hand, since $z + 2+\sqrt{5} \leq 6$ and $|b_t| \leq 1$, $\mu(b_t) \geq -7/2 - 6n^{-1}$. If $|b_0-2\ep| \leq n^{-1}$, $\tau_2 \leq T$ and $b_{\tau_2} \leq \ep$ then
\begin{eqnarray*}
-M_{\tau_2}(b^2) &\geq & (2\ep - n^{-1})^2 - \ep^2 - (7/2 + 6n^{-1})T \\
&\geq & 3\ep^2 - 4\ep n^{-1} - (7/2 + 6n^{-1})T\\
\end{eqnarray*}
If $n \geq 4/\ep$ then $4\ep n^{-1} \leq \ep^2$. Taking $\lambda = n/64$, $c\lambda =1$. If $n\geq 6$ then $(7/2)+6n^{-1} \leq 5$. Letting $T = \ep^2/6$,
$$-M_{\tau_2}(b^2)-c\lambda {\tau_2} \geq 2\epsilon^2 -6T \geq \epsilon^2,$$
and taking $a=\epsilon^2$ and $\bullet = -$, it follows that
\begin{equation}\label{eq18}
\P(\tau_2 \leq \ep^2/6\,\,\hbox{and}\,\,|b_{\tau_2}| \leq \ep \mid |b_0-2\ep| \leq n^{-1}) \leq e^{-\ep^2n/64}
\end{equation}
The result follows by stopping the process each time $|b-2\epsilon| \leq n^{-1}$, using \eqref{eq17} and \eqref{eq18}, then using the Markov property and taking a union bound while noting that $e^{-\ep^2n/64} \leq e^{-\ep^4n/8}$ for $\ep \leq 1/2\sqrt{2} \leq 1/4$.
\end{proof}

\begin{lemma}
\label{lemma2}
As in Lemma \ref{lemma1}, let $b_t = z_t-z^*$,
$$\tau = \inf\{t:u_t \geq 2\ep\}\quad\hbox{and}\quad\tau_1 = \tau\wedge\inf\{t:|b_t| \geq 3\ep\}.$$
Let
$$c_1 = z^*-3\epsilon \quad\hbox{and}\quad c_2 = z^*+3\epsilon.$$
For $C>0$, if $\ep \leq \min(C^{1/2}/6,1/60)$ and $n\geq \max(20,64(1+C^{1/2})^2)$ then
$$\P(\tau_1 > 25(3+C) + (2c_1)^{-1}\log n) \leq e^{-n^{1/2}/1600(1+C^{1/2})^2} + e^{-c_1C/18}$$
Also, for $C_2>0$,
\begin{eqnarray*}
&& \P( \, \tau_1 \leq (2c_2)^{-1}(\log n - C_2) \quad\hbox{and}\quad u_{\tau_1} \geq 2\ep \ \mid \ u_0 = C_2n^{-1/2} \, ) \\
&& \leq e^{-C_2+2/c_1 + n^{-1/2}} + (2\ep)^{-1}C_2e^{-C_2/2}
\end{eqnarray*}
\end{lemma}
\begin{proof}
Notice that $\mu(u) = uz + \rho_1\1(u=0) + \rho_2\mathbf{1}(u\leq n^{-1})$ where $0 \leq \rho_1,\rho_2\leq 1$. Using transitions 5 and 6 from Table \ref{tab1},
$$\mu(u^2) = 2u\mu(u) + \sum_i (q_i \cdot \Delta_i^2)(u) \geq 2u^2z + 4n^{-1}z(z-n^{-1})$$
If $|z-z^*| < 3\ep$ and $3\ep,n^{-1}\leq 1/20$, then since $z^*>1/5$, it follows that $z-n^{-1} > 1/10$ and $z(z-n^{-1}) > 1/100$, so that $\mu(u^2) \geq n^{-1}/25$.\\

For $C>0$, let $\tau_3= \tau_1\wedge \inf\{t:u^2 \geq Cn^{-1}\}$. If $u < C^{1/2}n^{-1/2}$ then since $c_\Delta(u) \le 2n^{-1}$,
\begin{eqnarray*}
|\Delta_i(u^2)| & \leq & 2u|\Delta_i(u)| + |\Delta_i(u)|^2 \\
& \leq & 4C^{1/2}n^{-3/2} + 4n^{-2} \leq C_1n^{-3/2} \quad\hbox{with}\quad C_1 = 4(1+C^{1/2})
\end{eqnarray*}
so we can take $\lambda \in (0,n^{3/2}/2C_1)$ and $c = C_1^2n^{-5/2}$. From the bound on the jump size, we find that $u_{\tau_3}^2 \leq Cn^{-1}+C_1n^{-3/2}$. If $\tau_3>T$ then noting $u_0^2 \geq 0$ it follows that
$$-M_T(u^2) \geq -Cn^{-1}-C_1n^{-3/2} + n^{-1} T/25$$
Taking $T=25(3+C)$ and $\lambda = C_1^{-2}n^{3/2}/50$, $c\lambda T = 2n^{-1}$ and $n^{-1}T/25 = (3+C)n^{-1}$, so
$$-M_T(u^2) -c\lambda T\geq n^{-1} - C_1n^{-3/2}.$$
If $n \geq (2C_1)^2$ this is at least $n^{-1}/2$, so taking $a=n^{-1}/2$ and $\bullet = -$ gives the estimate
$$\P(\tau_3 > 25(3+C)) \leq e^{-C_1^{-2}n^{1/2}/100}$$
Next, let $c_1 = z^*-3\ep$ and define $h$ by $h_t = e^{-c_1t}u_t$ so that $\mu_t(h) \geq 0$ and $-M_t(h) \geq u_0 - h_t$ for $t<\tau_1$. Here we take $c=c(t)$ that depends on time, such that $\sigma^2_t(h) \le c(t)$, and use the more general inequality $\langle h \rangle_t \le \int_0^t c(s)ds$. Since $|\Delta_i(u)| \leq 2n^{-1}$ for any $u,i$, $|\Delta_i(h_t)| \leq 2e^{-c_1t}n^{-1} \leq 2n^{-1}$, so we can take $\lambda \in (0,n/4)$ and $c(t) = 4n^{-1}e^{-2c_1 t}$.
 If $u_0 \geq C^{1/2}n^{-1/2}$ and $\tau_1 > T$ then
 $$-M_T(h) - \lambda \int_0^T c(t)dt \geq C^{1/2}n^{-1/2} - 2e^{-c_1T}\epsilon - 4n^{-1}\lambda\int_0^Te^{-2c_1t}dt.$$

Letting $T = (2c_1)^{-1}\log n$ and bounding the integral by $1/2c_1$ we obtain
$$-M_T(h) - \lambda \int_0^T c(t)dt \geq C^{1/2}n^{-1/2} - 2\epsilon n^{-1/2} - 4n^{-1}\lambda/2c_1.$$
If $C \leq (3/2c_1)^2n$, then taking $\lambda = C^{1/2}n^{1/2}c_1/6$ and $\ep \leq C^{1/2}/6$, and taking $a = C^{1/2}n^{-1/2}/3$ and $\bullet = -$ we find that
$$\P(\tau_1 > (2c_1)^{-1}\log n \mid u_0 \geq C^{1/2}n^{-1/2}) \leq e^{-c_1C/18}$$
To get a matching lower bound on $\tau_1$ we need an upper bound on $\mu(u)$. If $u>n^{-1}$ and $|z-z^*| \leq 3\ep$ then letting $c_2 = z^*+3\ep$, $\mu(u) \leq c_2u$. Let $\tau_4 = \tau_1\wedge \inf\{t:u_t \leq n^{-1}\}$.
 If $u_0 = C_2n^{-1/2}$ for $C_2>0$ and $u_{\tau_4} \leq n^{-1}$ then as before, for $\lambda \leq n/4$ we find
 $$-M_{\tau_4}(h) - \lambda \int_0^{\tau_4}c(t)dt \geq C_2n^{-1/2} - n^{-1} - 2\lambda n^{-1} /c_1$$
If $u_0 > n^{-1}$ then $\tau_4 \neq \tau_1$ if and only if $u_{\tau_4} \leq n^{-1}$. Taking $\lambda = n^{1/2}$, $a = C_2n^{-1/2} - n^{-1} - 2n^{-1/2}/c_1$ and using $\bullet = -$ we find that
$$\P( \, \tau_1 \neq \tau_4 \ \mid \ u_0 = C_2n^{-1/2} \,) \leq e^{-C_2+2/c_1 + n^{-1/2}}$$
On the other hand, $\mu(u_t) \leq c_2u_t$ for $t < \tau_4$, so $s_t = e^{-c_2(t \wedge \tau_4)}u_{t \wedge \tau_4}$ is a supermartingale. Using non-negativity of $s_t$, the fact that $t\mapsto e^{-c_2t}$ is non-increasing, and optional stopping,
\begin{eqnarray*}
\E[\, u_{\tau_4} \, ; \, \tau_4 \leq T \, ] &=& e^{c_2T}\E[\, e^{-c_2T}u_{\tau_4}\, ; \, \tau_4 \leq T \, ] \\
& \leq & e^{c_2T}\E[\, e^{-c_2\tau_4}u_{\tau_4} \, ] \leq e^{c_2T}\E[\, u_0 \,]
\end{eqnarray*}
Using Markov's inequality,
$$\P( \, \tau_4 \leq T \quad\hbox{and} \quad u_{\tau_4} \geq 2\epsilon \ \mid \ u_0 = C_2n^{-1/2} \, ) \leq (2\ep)^{-1}e^{c_2T}C_2n^{-1/2}$$
Letting $T = (2c_2)^{-1}(\log n - C_2)$, this is at most $(2\ep)^{-1}C_2e^{-C_2/2}$. The second statement then follows from a union bound.
\end{proof}

\begin{proposition}
\label{prop1}
 Let $\tau = \inf\{t:u_t \geq 2\ep\}$ as in Lemma \ref{lemma1},\ref{lemma2}. Then for any $\alpha>0$, there is $\epsilon_0>0$ so that for $\epsilon \in (0,\ep_0]$,
 \begin{eqnarray*}
 \lim_{n\to\infty}\sup_{(u,z)}\P_{(u,z)}(\tau > ((2z^*)^{-1} + \alpha) \log n \, )=0 \\
 \lim_{n\to\infty}\sup_{(u,z)}\P_{(u,z)}(\tau > ((2z^*)^{-1} - \alpha) \log n \, )=1
 \end{eqnarray*}
\end{proposition}
\begin{proof}
Given $\alpha>0$, let $0<\epsilon_0 \leq 1/60$ be small enough that
$$(2(z^*-3\ep_0))^{-1} \leq (2z^*)^{-1} + \alpha/2 \quad\hbox{and}\quad (2(z^*+3\ep_0))^{-1} \geq (2z^*)-\alpha/2.$$
Recall that $b_t = z_t-z^*$ and let $\tau_0 = \tau\wedge\inf\{t:|b_t|\leq \ep\}$ as in Lemma \ref{lemma1}. Fix $\epsilon  \in (0,\epsilon_0]$. Using the first result of Lemma \ref{lemma1},
$$\P(\,\tau_0 < 2/\ep^2\,) = 1-o(1)$$
If $\tau = \tau_0$ and $\tau_0<2/\ep^2$ then in particular, $\tau \leq ((2z^*)^{-1}+\alpha) \log n$ for large enough $n$. If $\tau>\tau_0$ then $|b_{\tau_0}| < \ep$. Letting $N=n$ in the second result of Lemma \ref{lemma1} and using the strong Markov property,
$$\P(\,|b_{\tau_0 + t}| < 3\ep \quad\hbox{for all}\quad t \leq (\tau-\tau_0)\wedge(\epsilon^2n/6)\, \mid |b_{\tau_0}|\leq \ep \, ) = 1-o(1)$$
Then, letting $C=(\alpha/100)\log n$ in the first result of Lemma \ref{lemma2} and using again the strong Markov property,
$$\P( \, |b_{\tau_0 + t}| < 3\ep \quad\hbox{and}\quad u_{\tau_0+t} < 2\ep \quad\hbox{for all} \quad t \leq 75 + (\alpha/4)\log n + ((2z^*)^{-1}+\alpha/2)\log n \, )  = o(1)$$
If $n$ is large enough then $\epsilon^2n/6 > 75 + (\alpha/4)\log n + ((2z^*)^{-1}+\alpha/2)\log n$. Combining these results, we find that from any initial distribution, if $n$ is large enough then
$$\P(\,\tau \geq 2/\ep^2 + 75 + (\alpha/4)\log n + ((2z^*)^{-1}+\alpha/2)\log n\,) = o(1)$$
If $n$ is large enough then $2/\ep^2 + 75 \leq (\alpha/4)\log n$ and the first statement follows. For the second statement, recall that $\tau_1 = \tau \wedge \inf\{t:|b_t| \geq 3\ep\}$, and let $C_2 = (\alpha/4c_2)\log n$ to find that
$$\P(\, \tau_1 \leq ((2c_2)^{-1} - \alpha/2)\log n\quad\hbox{and}\quad u_{\tau_1} \geq 2\ep \ \mid \ u_0 = C_2n^{-1/2}\,) = o(1)$$
By definition, either $u_{\tau_1} \geq 2\ep$ or $|b_{\tau_1}| \geq 3\ep$. Combining with the second result of Lemma \ref{lemma1}, if $\epsilon^2n/6 > ((2c_1)^{-1}-\alpha/2)\log n$ then
$$\P(\, \tau \leq ((2c_2)^{-1} - \alpha/2)\log n \ \mid \ u_0 = C_2n^{-1/2}\quad\hbox{and}\quad |b_0|<\ep\,) = o(1)$$
and the second statement follows.
\end{proof}

 Next we show that if $u_0 \geq 2\ep$ then there is a good chance $u_t \geq \ep$ for as long as we need.
\begin{lemma}
\label{lemma3}
If $n>1/\ep$ then for $T>0$,
$$\P(\inf_{t < T}u_t <\ep \mid u_0 \geq 2\ep) \leq e^{-\ep^2n/16T}$$
\end{lemma}
\begin{proof}
We know that if $u>n^{-1}$ then $\mu(u) = uz \geq 0$.
 Since $c_\Delta(u) \le 2n^{-1}$, we can take $\lambda \in (0,n/4)$ and $c = 4n^{-1}$.
 Let $\tau_0 = \inf\{t:u_t < \ep\}$. If $u_0 \geq 2\ep$ and $\tau_0 < T$ then $M_{\tau_0}(u)-c\lambda \tau_0 \geq \ep -4\lambda n^{-1}T$, so using $\bullet = +$,
$$\P(\tau_0<T \mid u_0\geq 2\ep) \leq e^{-\lambda(\ep-4\lambda n^{-1}T)}$$
Optimizing in $\lambda$ then gives the result.
\end{proof}

\begin{lemma}
\label{lemma4}
Let $\tau = \inf\{ \, t \colon u_t < \ep\quad\hbox{or}\quad u_t > 1-\ep \, \}$ and let
$$\tau_2 = \tau\wedge \inf\{t:z_t > \ep/4\}\quad\hbox{and}\quad \tau_3 = \tau\wedge \inf\{t:z_t < \ep/12\}.$$
If $\ep \leq 1/4$ and $n\geq 1/\ep$ then
$$\P( \, \tau_2 > 20 \, ) \leq e^{-\ep^2n/128},$$
and for integer $N>0$,
$$\P( \, \tau_3 \leq \ep N/48 \quad\hbox{and}\quad \, z_{\tau_3} < \ep/12 \ \mid \ z_0 > \ep/4 \, ) \leq 2Ne^{-n\ep^2/192+1/4}.$$
\end{lemma}
\begin{proof}
Since $c_\Delta(z) \le 2n^{-1}$, we can take $\lambda \in (0,n/4)$ and $c = 4n^{-1}$. Recall that
$$\mu(z) = (1/2)(1-u^2-4z-z^2) + 2zn^{-1}$$
so if $u\leq 1-\ep$, $z\leq \ep/4$ and $\ep\leq 1/4$ then $\mu(z) \geq \ep/4$. Minding the jump size, $z_{\tau_2} \leq \ep/4+2n^{-1}$. If $\tau_2 > T$ and $\epsilon/4 \geq 4n^{-1}\lambda$ then
\begin{eqnarray*}
-M_{\tau_2}(z) - c\lambda \tau_2  & \geq & -\ep/4 - 2n^{-1} + (\ep/4)T - 4n^{-1} \lambda T \\
&=& \ep(T(1-16\lambda(\ep n)^{-1})-1-8(\ep n)^{-1})/4
\end{eqnarray*}
Taking $T=20$ and $\lambda = \ep n/32$, if $n \geq 1/\ep$ we have the lower bound $\ep/4$. Taking $a=\ep/4$ and $\bullet = -$ gives the first statement. Now, let $\tau_4 = \tau\wedge\inf\{t:|z_t -\ep/6|\geq \ep/12\}$. Suppose that $|z_0-\ep/6| \leq n^{-1}$, then $\mu(z_t) \geq \ep/4$ for $t<\tau_4$. If $z_{\tau_4} \leq \ep/12$ then
$$-M_{\tau_4}(z) - c\lambda \tau_4 \geq \epsilon/6 - n^{-1}-\epsilon/12 + (\epsilon/4 - 4\lambda n^{-1})\tau_4.$$
Taking $\lambda = \ep n/16$ and $a = \ep/12 - n^{-1}$, $\P(z_{\tau_4} \leq \ep/12) \leq e^{-\ep^2 n/192 + \ep/16}$.
 On the other hand, if $z_{\tau_4} \geq \ep/4$ and $\tau_4\leq T$ then since $\mu(z) \leq 1/2+2n^{-1}$,
$$M_{\tau_4}(z) - c\lambda \tau_4 \geq \ep/12-n^{-1}- T(1/2+2n^{-1}+ 4\lambda n^{-1}))$$
Taking $\lambda = n/4$, $T=\ep/48$ and $a = \ep/24-n^{-1}$, if $n \geq 4$ then
$$\P( \, \tau_4 \leq \ep/48\quad\hbox{and}\quad z_{\tau_4} \geq \ep/4 \, ) \leq e^{-\ep n/96 + 1/4}.$$
Combining these,
$$\P( \, \tau_4 \leq \ep/48\quad\hbox{or}\quad z_{\tau_4} \leq \ep/12 \, ) \leq e^{1/4}(e^{-\ep^2 n/192} + e^{-\ep n/96})$$
The result follows by stopping the process each time $|z - \ep/6| \leq n^{-1}$, using the strong Markov property, taking a union bound, and using $e^{-\ep n/92} \leq e^{-\ep^2 n/192}$.
\end{proof}

\begin{lemma}
\label{lemma5}
 Let $\tau = \inf\{t:u_t < \ep \,\,\hbox{or}\,\, u_t > 1-\ep\}$ as in Lemma \ref{lemma4} and let $\tau_5 = \tau\wedge \inf\{t:z_t < \ep/12\}$. 
 Then,
 $$\P(\tau_5 > 48/\ep^2) \leq e^{-\ep^2n/96}$$
\end{lemma}
\begin{proof}
If $u \geq \ep$ and $z\geq \ep/12$ then $\mu(u) \geq \ep^2/12$.
 Since $c_\Delta(u) \leq 2n^{-1}$, we can take $\lambda \leq n/4$ and $c=4n^{-1}$.
 If $\tau_5 > T$ then since $u_T - u_0\leq 1$, if $\ep^2/12 \geq 4\lambda n^{-1}$ then
$$-M_{\tau_5}(u) - c\lambda \tau_5 \geq -1 + T(\ep^2/12 - 4\lambda n^{-1})$$
Taking $\lambda = -\ep^2n/96$, $T= 48/\ep^2$, $a=1$ and $\bullet =-$, the result follows.
\end{proof}

\begin{proposition}
\label{prop2}
Let $\tau = \inf\{t:u_t \leq \ep\,\,\hbox{or}\,\,u_t \geq 1-\ep\}$, as in Lemma \ref{lemma4}. Then for any $\alpha>0$ and $\epsilon \in (0,1/4)$,
$$\lim_{n\to\infty}\sup_{(u,z):u \geq 2\ep}\P_{(u,z)}( \, u_{\tau} \leq \ep\,\,\hbox{or}\,\,\tau > \alpha\log n \, ) = 0$$
\end{proposition}
\begin{proof}
Taking $T = n^{1/2}$ in Lemma \ref{lemma3},
$$\P( \, \inf_{t \leq n^{1/2}}u_t \leq \ep \ \mid \ u_0 \geq 2\ep \, ) = o(1)$$
Let $\tau_2 = \tau \wedge \inf\{t \colon z_t > \ep/4\}$ as in Lemma \ref{lemma4}. Using the first result of Lemma \ref{lemma4},
$$\P( \, \tau_2 \leq 20 \, ) = 1 - o(1)$$
Letting $N = n$ in the second result of Lemma \ref{lemma4} and using the strong Markov property,
$$\P( z_{\tau_2 + t} \geq \ep/12 \quad\hbox{for all}\quad t \leq (\tau-\tau_2) \wedge (\ep n/ 48) \ \mid \ z_{\tau_2} > \ep/4 \, ) = 1 - o(1)$$
Then, using the strong Markov property and the result of Lemma \ref{lemma5},
$$\P( z_{\tau_2 + t} \geq \ep/12 \quad\hbox{and}\quad u_{\tau_2 + t} \in [\epsilon,1-\epsilon]\quad\hbox{for all}\quad t \leq 48/\ep^2 \,) = o(1)$$
If $n$ is large enough then $\min(n^{1/2},\,\ep n/48,\,\alpha \log n) \geq 20 + 48/\ep^2$. Combining the estimates gives the result.
\end{proof}

\begin{lemma}
\label{lemma6}
Let $v = \max(x,y),\,w = \min(x,y)$ and define
$$\tau=\inf\{t:w_t=z_t=0 \quad\hbox{or}\quad 2w_t+z_t \geq 2\ep\}$$
For any $T>0$ and $\ep\leq 1/4$,
$$\P(\, 2w_{\tau}+z_{\tau} \geq 2\ep \,\,\hbox{and}\,\,\tau \leq T \mid 2w_0+z_0 \leq \ep \, ) \leq e^{-\ep^2n/16T}$$
Also, if $n\geq 4/\ep$ then for $c>0$,
$$\P( \, w_{\tau}=z_{\tau}=0\,\,\hbox{and}\,\,\tau \leq C \log n \ \mid \ 2w_0+z_0 \geq \ep-2n^{-1} \, ) \leq 12\ep^{-1}n^{-1 + (1+13\ep/2)C}$$
\end{lemma}
\begin{proof}
We have $v = u + w$ and $2w+z = 1-u$. Recall that if $u>n^{-1}$ then $\mu(u) = uz$, so $\mu(2w+z) = (w-v)z$. Since $w\leq v$, $\mu(2w+z) \leq 0$. Since $c_\Delta(2w+z) \leq 2n^{-1}$, we can take $\lambda \leq n/4$ and $c = 4n^{-1}$.
 If $2w_0+z_0 \leq \ep$, $2w_{\tau}+z_{\tau} \geq 2\ep$ and $\tau \leq T$ then $M_{\tau}(2w+z) - c\lambda \tau \geq \ep - 4\lambda n^{-1}T$. Taking $\lambda = n/8T$ and $\bullet = +$, the first statement follows. On the other hand, we have always $v\leq 1$, and if $t<\tau$ then $w_t \leq \ep$, $z_t \leq 2\ep$ and $v_t \geq 1-2\ep$. Looking to Table \ref{tab1}, ignoring the $5^{th}$ and $8^{th}$ transitions, ignoring some increases, and bounding the rates in the right direction it is easy to check that for $t \leq \tau$, $(w_t,z_t)$ dominates the process $(\tilde{w}_t,\tilde{z}_t)$ with initial value $(w_0,z_0)$ and the following transitions:

\vspace{0.15in}
\begin{tabular}{l | l l l l l l}

$n\Delta_i(\tilde{w})$ & -1 & 0 &0 & 0 & 0 & -1 \\
\hline
$n\Delta_i(\tilde{z})$ & 0 & -1 & 1 & -1 & -2 & 0 \\
\hline
$n^{-1}q_i$ & $\tilde{w}\ep$ & $3\ep \tilde{z}/2$ & $(1-2\ep)\tilde{z}/2$ & $3\tilde{z}/2$ & $2\ep\tilde{z}$ & $\tilde{w}$ \\
\end{tabular}\\
\vspace{0.15in}

Note the transition rates are linear. We easily compute
$$\mu(\tilde{w}) = -(1+\ep)\tilde{w} \quad\hbox{and}\quad \mu(\tilde{w}^2) = (1+\ep)(-2\tilde{w}^2 + n^{-1}\tilde{w})$$
so that if $w_0$ is deterministic, we solve to obtain
$$\E[\tilde{w}_t] = e^{-(1+\ep)t}w_0\quad \hbox{and}\quad \E[\tilde{w}_t^2] = e^{-2(1+\ep)t}w_0^2 + n^{-1}w_0e^{-(1+\ep)t}(1-e^{-(1+\ep)t})$$
Combining, $\var(\tilde{w}_t) \leq n^{-1}w_0e^{-(1+\ep)t}$ and so
$$\P(\tilde{w}_t=0) \leq \P(|\tilde{w}_t - \E[\tilde{w}_t]| \geq \E[\tilde{w}_t]) \leq \frac{\var(\tilde{w}_t)}{(\E[\tilde{w}_t])^2} \leq (w_0 n)^{-1}e^{(1+\ep)t}$$
Similarly,
$$\mu(\tilde{z}) = -(1+13\ep/2)\tilde{z} \quad\hbox{and}\quad \mu(\tilde{z}^2) = -2(1+13\ep/2)\tilde{z}^2 + n^{-1}(2+17\ep/2)\tilde{z}$$
so that if $z_0$ is deterministic,
$$\E[\tilde{z}_t] \leq e^{-(1+13\ep/2)t}z_0\quad \hbox{and}\quad \var(z_t) \leq n^{-1}z_0\frac{2+17\ep/2}{1+13\ep/2}e^{-(1+13\ep/2)t}$$
and since the above fraction is at most $2$,
$$\P(\tilde{z}_t=0) \leq 2(z_0 n)^{-1}e^{(1+13\ep/2)t}$$
If $2w_0+z_0 \geq a$ then $\max(w_0,z_0) \geq a/3$, so for $T>0$
\begin{eqnarray*}
&& \P(\sup_{t \leq T}(2w_t+z_t) \leq 2\ep,\,\,z_T=w_T=0 \mid 2w_0+z_0 \geq a) \\
&& \leq \max(\P(\tilde{w}_T=0 \mid w_0 \geq a/3),\P(\tilde{z}_T=0 \mid z_0\geq a/3))
\end{eqnarray*}
Letting $a = \ep-2n^{-1}$ and $T=C\log n$, if $n\geq 4/\ep$ then $a \geq \ep/2$ and the second statement follows.
\end{proof}

\begin{lemma}
\label{lemma7}
Let $v,w$ and $\tau$ be as in Lemma \ref{lemma6}. If $\ep < 1/6$, $n \geq 1/\ep$ and $C>0$ then
$$\P(\, \tau>(1-6\ep)^{-1}(1-2\ep)^{-1}(\log n +C) \ \mid \ 2w_0+z_0 \leq \ep \, ) \leq e^{-C}/4$$
\end{lemma}
\begin{proof}
Let $\psi_t = (w_t,z_t)^{\top}$. We may assume $2w+z \leq 2\ep$ so that $v=1-(w+z)\geq 1-2\ep$. Define the $2\times 2$ matrices $Q = (-1,\, 0\,;\,2 ,\,-1)$ and $B = (1,\,1\,;\, 1,\,1)$. Computing,
\begin{eqnarray*}
\mu(w) &=& - wv + (w+z)z - n^{-1}z \\ \nonumber
\mu(z) &=& (2w -z)v - (w+2z)z + 2n^{-1}z \nonumber
\end{eqnarray*}
So, if $n \geq 1/\ep$ then letting $\delta = 2\ep/(1-2\ep)$,
$$\mu(\psi) \leq v(Q+\delta B)\psi$$
Time-change by $v^{-1}$ so that $\mu(\psi) \leq Q_{\delta}\psi$ with $Q_{\delta} := Q+\delta B$. Let $s_t = e^{-Q_{\delta}t}\psi_t$ so that $s_{t\wedge \tau}$ is a non-negative supermartingale. Using non-negativity and optional stopping, we see that
$$e^{-Q_{\delta}t}\E[\psi_t\,;\,\tau>t] = \E[s_t\,;\,\tau>t] = \E[s_{t\wedge\tau},\,;\tau>t] \leq \E[s_{\tau}]\leq \E[s_0] = \E[\psi_0]$$
and so
$$\P(\tau>t) \leq \P(\tau>t,\,z_t \geq n^{-1}) \leq n\E[z_t\,;\,\tau>t] \leq n|e^{Q_{\delta}t}|\max(w_0,z_0)$$
Now, $Q_{\delta}$ has eigenvalues $-1+\delta \pm 2\delta^2$ and corresponding eigenvectors $(1,\pm 2\delta)^{\top}$. Thus $Q_{\delta} = SAS^{-1}$ with $S = (1,\,1\,;\,2\delta,\,-2\delta)$ and $S^{-1} = (1/2\delta,\,1/4\,;\,1/2\delta,\,-1/4)$. If $\delta\leq 1$, then $|S| \leq 1$, $|S^{-1}| \leq (2\delta)^{-1}$ and $|A|\leq -1 + \delta + 2\delta^2$ and so
$$|e^{Q_{\delta}t}| \leq |S||e^{At}||S^{-1}| \leq (2\delta)^{-1}e^{-(1-\delta-2\delta^2)t}$$
If $\ep\leq 1/6$ then $1-2\ep\geq 2/3$, $\delta \leq 3\ep \leq 1/2$ and $1-\delta(1+2\delta) \geq 1 - 6\ep$. If $w_0,z_0 \leq \ep$ then
$$|e^{Q_{\delta}t}|\max(w_0,z_0) \leq (\ep/2\delta)e^{-(1-6\ep)t} = (1-2\ep)e^{-(1-6\ep)t}/4$$
Recalling the time change and noting $v^{-1} \leq (1-2\ep)^{-1}$, then letting $t= (1-6\ep)^{-1}(1-2\ep)^{-1}(\log n +C)$ and using the fact that $1-2\ep \leq 1$ gives the result.
\end{proof}

\begin{proposition}
\label{prop3}
 Let $\tau = \inf\{t:u_t=1\,\,\hbox{or}\,\,u_t \leq 1-2\ep\}$. For $\alpha>0$, there is $\epsilon_0>0$ so that for $\ep \in (0,\ep_0]$,
 $$\lim_{n\to\infty}\sup_{(u,z):u \geq 1-\ep}\P_{(u,z)}(\tau > (1+\alpha)\log n \quad\hbox{or}\quad u_{\tau} \leq 1-2\ep) = 0$$
 and
 $$\lim_{n\to\infty}\inf_{(u,z):|u-(1-\ep+n^{-1})| \leq n^{-1}}\P_{(u,z)}(\tau > (1-\alpha)\log n) = 1$$
\end{proposition}
\begin{proof}
Since $2w+z = 1-u$, the above definition of $\tau$ agrees with the one used in Lemmas \ref{lemma6}-\ref{lemma7}. Given $\alpha>0$, $0<\epsilon_0 <1/6$ be small enough that for $\epsilon \in (0,\ep_0]$,$$(1-6\ep)^{-1}(1-2\ep)^{-1} \leq 1 + \alpha/2\quad\hbox{and}\quad(1+13\epsilon/2)(1-\alpha) \leq 1-\alpha/2.$$ Letting $T = n^{1/2}$ and using the first result of Lemma \ref{lemma6},
\begin{equation}\label{eq19}
\P( \, \tau \leq n^{1/2}\quad\hbox{and}\quad u_{\tau} \leq 1-2\ep \ \mid \ u_0 \geq 1-\ep \, ) = o(1)
\end{equation}
On the other hand, letting $C=(\alpha/2)\log n$ and using the first result of Lemma \ref{lemma7},
$$\P( \, \tau > (1+\alpha)\log n \ \mid \ u_0 \geq 1-\ep \, ) \leq n^{-\alpha/2}/4 = o(1)$$
Since $n^{1/2} > (1+\alpha) \log n $ for $n$ large enough, the first statement follows. For the second statement, letting $C = 1-\alpha$ in Lemma \ref{lemma6},
$$\P( \, u_{\tau} = 1 \quad\hbox{and} \quad \tau\leq ( 1-\alpha )\log n \ \mid \ u_0 \in [1-\ep,1-\ep+2n^{-1}]\,) \leq 12\epsilon^{-1}n^{-\alpha/2} = o(1)$$
Combining with \eqref{eq19}, the second statement follows.
\end{proof}

\begin{proof}[Proof of Theorem \ref{thm:finalphase}]
Let $\tau = \inf\{t:u_t=1\}$. Recall that $z^* = -2+\sqrt{5}$. To show the upper bound, for any $\alpha>0$ take $\ep>0$ small enough to satisfy all conditions, thenapply Propositions \ref{prop1}, \ref{prop2} and \ref{prop3} in sequence, stopping the process when $u_t \geq 2\ep$ and $u_t \geq 1-\ep$, to find that
$$\lim_{n\to\infty}\sup_{(u,z)} \P_{(u,z)}(\, \tau > (1 + (2z^*)^{-1} + 3\alpha)\log n \, ) = 0$$
To show the lower bound, in Proposition \ref{prop1} start from $(u,z)$ achieving the supremum, which is a maximum since the state space is finite. Apply the result of Proposition \ref{prop1}. Then, stop the process when $|u_t - (1-\ep - n^{-1})| \leq n^{-1}$, which occurs before $\tau$ since $u_t$ has jumps of size at most $2n^{-1}$. Apply Proposition \ref{prop3}. Combining the two, conclude that for any $\alpha>0$,
$$\lim_{n\rightarrow\infty}\sup_{(u,z)} \P_{(u,z)}(\,\tau > (1 + (2z*)^{-1} - 2\alpha)\log n \, ) =1$$
\end{proof}


\section*{Acknowledgements}
The author wishes to thank Nicolas Lanchier for suggesting the model, and for many helpful conversations while working on the paper.

\section*{Appendix}
\textbf{Miscellaneous estimates.}
\begin{enumerate}
\item Since
$$\frac{d}{dx}(x^{1+a-\beta}e^{-cx^{\beta}}) 
= ((1+a-\beta)x^{-\beta} - c)x^ae^{-cx^{\beta}}$$
and $t\mapsto c-(1+a-\beta)t^{-\beta}$ increases with $t$, if $c-(1+a-\beta)x^{-\beta}>0$ we have the upper bound
\begin{equation}\label{eq:exp-asym}
\begin{array}{rcl}
\int_x^{\infty}t^ae^{-ct^{\beta}}dt & \le & (c - (1+a-\beta)x^{-\beta})^{-1}\int_x^{\infty}(c - (1+a-\beta)t^{-\beta})t^ae^{-ct^{\beta}}dt \\
& = & (c-(1+a-\beta)x^{-\beta})^{-1}x^{1+a-\beta}e^{-cx^{\beta}}.
\end{array}
\end{equation}

\item Factoring, using the fact that $|(1+\lambda)^{-1}\lambda^{1/2+\alpha}/2| \leq 1/2$ and $(1-x)^{-1} \leq 1 + 2x$ for $|x| \leq 1/2$, then using the fact that $\lambda^{1/2+\alpha} \leq (1+\lambda)^{1/2+\alpha}$,
\begin{equation}\label{eq:recip-bound}
\begin{array}{rcl}
(1+\lambda-\lambda^{1/2+\alpha}/2)^{-1} &=& (1+\lambda)^{-1}(1 -(1+\lambda)^{-1}\lambda^{1/2+\alpha}/2)^{-1} \\
& \leq & (1+\lambda)^{-1}(1 + 2(1+\lambda)^{-1}\lambda^{1/2+\alpha}/2) \\
&\leq & (1+\lambda)^{-1} + (1+\lambda)^{-3/2+\alpha}
\end{array}
\end{equation}
\end{enumerate}

\begin{lemma}\label{lem:mod-dev}
Let $X$ be Poisson with mean $\lambda$.
\begin{equation}\label{eq:poi-md}
\begin{array}{rl}
\hbox{For}\quad 0<x \le \lambda^{1/2}, & \P( X < \lambda - x \lambda^{1/2}) \le e^{-x^2/2} \quad\hbox{and}\\
& \P( X > \lambda + x\lambda^{1/2}) \le  e^{-x^2/3}.
\end{array}
\end{equation}
\end{lemma}

\begin{proof}
We have
$$\E[e^{\theta X}] = \sum_{k\ge 0}e^{\theta k}e^{-\lambda}\lambda^k/k!
= e^{-\lambda}\sum_{k \ge 0}(\lambda e^{\theta})^k/k! = \exp(\lambda(e^{\theta}-1)).$$
Also,
$$\P(e^{\theta X} \ge e^{\lambda\theta c}) = \begin{cases} \P(X \ge c\lambda) & \hbox{if} \ \theta>0 \\
\P(X \le c\lambda) & \hbox{if} \ \theta<0.\end{cases}$$
Using Markov's inequality,
$$\P(e^{\theta X} \ge e^{\lambda \theta c}) \le e^{-\lambda \theta c}\E[e^{\theta X}] = \exp(\lambda(e^{\theta}-1-\theta c))$$
Optimizing in $\theta$ gives $\theta=\log c$ which is positive for $c>1$ and negative for $c<1$, and
$$\gamma(c) = e^{\theta}-1-\theta c = c - 1 - c\log c.$$
Expanding $\gamma(1+\delta)$ in an alternating Taylor series around $\delta=0$,
$$\begin{array}{rcl}
\gamma(1+\delta) &\le & -\delta^2/2 + \delta^3/6\ \quad \hbox{for} \quad |\delta|<1,  \ \hbox{so} \\
&\le & \begin{cases} -\delta^2/2 & \hbox{for} \quad -1<\delta \le 0 \\
-\delta^2/3 & \hbox{for} \quad 0 \le \delta<1, \end{cases} 
\end{array}$$
using $\delta^3 \le \delta^2$ for $\delta \in [0,1)$ and $\frac{1}{2}+\frac{1}{6}=\frac{1}{3}$. \eqref{eq:poi-md} follows for $0<x<\lambda^{1/2}$ by letting $\delta = x\lambda^{-1/2}$. For $x=\lambda^{1/2}$ it follows by continuity of probability.
\end{proof}

\begin{lemma}\label{lem:poi-proc-md}
Let $(N_t)$ be a Poisson process with intensity $\lambda$. Fix $\alpha \in (0,1/2]$ and let
$$\begin{array}{rcccl}
\tau_1 &=& \sup\{t:N_t - \lambda t &\ge & \ \lambda^{1/2+\alpha}/2\} \quad \hbox{and} \\
\tau_2 &=& \sup\{t:N_t - \lambda t &\le & -\lambda^{1/2+\alpha}/2\}
\end{array}$$
denote the last passage time of $N_t$ above/below the curve $\lambda t \pm (\lambda t)^{1/2+\alpha}/2$, respectively. If $\lambda \ge 1$ and $t^{2\alpha} \ge 6$ then
$$\begin{array}{rcl}
\P(\tau_1 > t) &\le & 6t^{1-2\alpha}e^{-(\lambda t)^{2\alpha}/3} \quad \hbox{and}\\
\P(\tau_2 > t) &\le & 4t^{1-2\alpha}e^{-(\lambda t)^{2\alpha}/2}.
\end{array}$$
\end{lemma}

\begin{proof}
Let $f$ denote the function defined by $f(t) = \lambda t  + (\lambda t)^{1/2+\alpha}$. Using Lemma \ref{lem:mod-dev}, for each $t>0$,
$$\P(N_t > f(t)) \le e^{-(\lambda t)^{2\alpha}/3}.$$
Since $|f'(t)| \le 2\lambda$ for any $t\ge 0$, $f$ is Lipschitz with constant $2\lambda$. Using this and the fact that $t\mapsto N_t$ is non-decreasing,
$$\{\sup_{s \in [t-1,t]}N_s - f(s) > 2\lambda\} \subseteq \{N_t > f(t)\},$$
so taking a union bound over $t \in \{T+1,T+2,\dots\}$,
$$\P(\sup_{t \ge T}N_t - f(t) > 2\lambda) \le \sum_{k \ge 1}e^{-(\lambda (T+k))^{2\alpha}/3}
\le \int_T^{\infty}e^{-(\lambda t)^{2\alpha}/3}dt.$$
Using \eqref{eq:exp-asym},
$$\int_T^{\infty}e^{-(\lambda t)^{2\alpha}/3}dt \le (\lambda^{2\alpha}/3-(1-2\alpha))T^{-(2\alpha)})^{-1}T^{1-2\alpha}e^{-T^{2\alpha}/3}.$$
If $\lambda \ge 1$ and $T^{2\alpha} \ge 6$, this is at most $6T^{1/2-\alpha}e^{-(\lambda T)^{1/2+\alpha}/3}$. An analogous estimate applies for the lower bound, giving $4$ instead of $6$ and $1/2$ instead of $1/3$ in the exponent, when $\lambda \ge 1$ and $T^{2\alpha} \ge 4$.\\
\end{proof}

\begin{lemma}\label{lem:inhlin-drift}
Let $X$ be a non-decreasing quasi-absolutely continuous semimartingale on $\R_+$ with jump size at most $c$ and defined for $t<\zeta$, where $\zeta = \sup_{r>0} \inf\{t:X_t \ge r\}$ is the first time of explosion. Suppose that
\begin{equation}\label{eq:lindrift}
\mu_t(X) \le b(t) + \ell(t)X_t
\end{equation}
for some locally integrable non-nonegative deterministic functions $b(t)$, $\ell(t)$. Let $m(t) = \exp(\int_0^t \ell(s)ds)$ and let $Y_t = X_t/(X_0 m(t)) - \int_0^tb(s)/m(s)ds$ denote the rescaled process. Let $\zeta' = \zeta \wedge \inf\{t:m(t)=\infty\}$ and $\beta = \int_0^{\infty}b(t)/m(t)^2dt$, and assume $\beta<\infty$. Then, $\zeta \ge \zeta'$ and for $y\geq 2$,
$$\P(\sup_{t < \zeta'} Y_t \geq y) \leq \E[e^{-(y-2)X_0/4c(1+\beta)}].$$
\end{lemma}

\begin{proof}
First we treat the case $X_0=1$, so that $Y_t = X_t/m(t)- \int_0^t b(s)/m(s)ds$. Given $y>0$ define $\tau(y) = \inf\{t:Y_t \geq y\}$, and note that $\tau(y) < \zeta'$.
 Since $1/m(t) = e^{-\int_0^t\ell(s)ds}$, $(1/m(t))' = -\ell(t)/m(t)$, so using linearity of the drift and Lemma \ref{lem:qac-prod},
$$\mu(Y_t) \le (b(t) + \ell(t)X_t)/m(t) + X_t(-\ell(t)/m(t)) -b(t)/m(t) = 0,$$
which implies $Y^p \le 0$. Clearly $\sigma^2_t(Y) = (1/m(t))^2\sigma^2_t(X)$. 
Since $X$ is non-decreasing, it has finite variation, so in particular $X^c=0$, $X^m=X^d$ and $\langle X^m \rangle_t = (\sum_{s \le t}(\Delta X_s)^2)^p$.   In addition, $0 \le \Delta X_s \le c$, so $(\Delta X_s)^2 \le c\Delta X_s$. Using this and $\sum_{t \le s \le t+r}\Delta X_s \le X_{t+r}-X_t$, for any $t,r$,
$$\langle X^m \rangle_{t+r}-\langle X^m \rangle_t \le c(\sum_{t \le s \le t+r}\Delta X_s)^p \le c( X^p_t-X^p_r )$$
which implies $\sigma^2_t(X) \le c \mu_t(X)$. Using $\mu_t(X)\le b(t) + \ell(t)X_t = b(t) + \ell(t)m(t)Y_t$,
$$\sigma^2_t(Y) \leq (1/m(t))^2c\mu_t(X) = cb(t)/m(t)^2 + (c/m(t))\ell(t) Y_t.$$
Since $Y_t <y$ for $t<\tau(y)$,
$$\langle Y \rangle_{\tau(y)} \leq c\int_0^{\tau(y)} b(s)/m(s)^2ds + y c\int_0^{\tau(y)}\ell(s)/m(s)ds = c\beta(\tau(y)) + yc \alpha(\tau(y)),$$
the last equality defining $\alpha(t)$ and $\beta(t)$. Taking the antiderivative,
$$\alpha(t) = \int_0^t e^{-\int_0^s \ell(r)dr}\ell(s)ds = 1-e^{-\int_0^t \ell(s)ds} = 1-1/m(t) \leq 1 \quad\hbox{for all} \quad t \geq 0.$$
Since $Y_0=1$, $Y_{\tau(y)} \geq y$ and $Y^p \le 0$, it follows that for $\lambda>0$,
$$Y_{\tau(y)}-Y_0 - Y^p_{\tau(y)}-\lambda \langle Y \rangle_{\tau(y)} \geq y - 1 - \lambda c (\beta + y).$$
Using \eqref{eq:sm-est2} with $a=y-1-\lambda c(\beta + y)$, assuming $\lambda c \le 1/2$ we find
$$P(\sup_{t<\zeta'} Y_t \geq y) \le e^{-\lambda a}.$$
Optimizing $\lambda a$ gives $\lambda = (y-1)/(2c(y+\beta))$ and
$$\lambda a \ge (y-1)^2/(4cy(1+\beta/y)) \geq (y - 2)/(4c(1+\beta)),$$
and if $y \ge 1$ the assumption $c\lambda \le 1/2$ holds.
For general $X_0$, first condition on $X_0$ and apply the above to $X_t/X_0$, which has jump size $c/X_0$. Then, integrate over $X_0$ to obtain the result.\\

To see that $\zeta \ge \zeta'$, note that $\{\zeta \ge \zeta'\} \supset \bigcup_y \{\sup_{t < \zeta'}Y_t < y\}$ and that the above estimate implies the latter event has probability 1.
\end{proof}
\bibliographystyle{plain}
\bibliography{naming-game}
\end{document}